\documentclass[12pt]{article}
\usepackage{amsmath,amssymb,amsthm,fullpage,xypic,mathrsfs,url}

\theoremstyle{definition}
\newtheorem{theorem}[equation]{Theorem}
\newtheorem{lemma}[equation]{Lemma}
\newtheorem{proposition}[equation]{Proposition}
\newtheorem{corollary}[equation]{Corollary}

\newtheorem{example}[equation]{Example}
\newtheorem{question}[equation]{Question}

\newtheorem{remark}[equation]{Remark}

\newcommand{\bC}{\mathbf{C}}
\newcommand{\bA}{\mathbf{A}}
\newcommand{\Spec}{\operatorname{\mathsf{Spec}}}
\newcommand{\HP}{\operatorname{HP}}
\newcommand{\Spf}{\operatorname{\mathsf{Spf}}}
\newcommand{\cO}{\mathcal{O}}
\newcommand{\bZ}{\mathbf{Z}}
\newcommand{\bH}{\mathbf{H}}

\newcommand{\Aut}{\operatorname{Aut}}
\newcommand{\caD}{\mathcal{D}}
\newcommand{\onto}{\twoheadrightarrow}

\newcommand{\Lie}{\operatorname{Lie}}

\newcommand{\mfg}{\mathfrak{g}}
\newcommand{\SL}{\operatorname{\mathsf{SL}}}
\newcommand{\Proj}{\operatorname{Proj}}
\newcommand{\Sym}{\operatorname{Sym}}
\newcommand{\img}{\operatorname{im}}
\newcommand{\Hom}{\operatorname{Hom}}
\newcommand{\End}{\operatorname{End}}
\newcommand{\Id}{\operatorname{Id}}
\newcommand{\Mat}{\operatorname{Mat}}
\newcommand{\rk}{\operatorname{rk}}
\newcommand{\pr}{\operatorname{pr}}
\newcommand{\Sp}{\operatorname{Sp}}
\newcommand{\GL}{\operatorname{GL}}
\newcommand{\Stab}{\operatorname{Stab}}
\newcommand{\Weyl}{\operatorname{Weyl}}
\newcommand{\Eu}{\operatorname{Eu}}
\newcommand{\gr}{\operatorname{gr}}
\newcommand{\ad}{\operatorname{ad}}
\begin{document}
\author{Travis Schedler\thanks{Department of Mathematics, University of Texas at Austin; 1
  University Station C1200, Austin, TX 78712-0257;
  trasched@gmail.com}} 
\title{Equivariant slices for symplectic
  cones}

\date{\today}
\maketitle

\begin{abstract}
  The Darboux-Weinstein decomposition is a central result in the
  theory of complex Poisson (degenerate symplectic) varieties, which gives a
  local decomposition at a point as a product of the formal
  neighborhood of the symplectic leaf through the point and a formal
  slice.

  Recently, conical symplectic resolutions, and more generally,
  Poisson cones, have been very actively studied in representation
  theory and algebraic geometry.  This motivates asking for a
  $\bC^\times$-equivariant version of the Darboux-Weinstein
  decomposition.

  In this paper, we develop such a theory, prove basic results on
  their existence and uniqueness, study examples (quotient
  singularities and hypertoric varieties), and applications to
  noncommutative algebra (their quantization). We also pose some
  natural questions on existence and quantization of
  $\bC^\times$-actions on slices to conical symplectic leaves.
\end{abstract}

\section{Introduction}
A \emph{conical} complex symplectic or Poisson variety is one which is
equipped with a contracting $\bC^\times$-action which does not
necessarily preserve the symplectic form or Poisson bracket, but
rather rescales it.

Recently, the theory of conical complex symplectic resolutions
($\bC^\times$-equivariant resolutions of singularities of a cone by a
symplectic variety), and more generally, Poisson cones, has been
widely studied not only in mathematics, but also in physics, and has
applications and connections to representation theory, symplectic
geometry, quantum cohomology, mirror symmetry, equivariant cohomology,
and other subjects (see, e.g., \cite{BPW-qcsr,BLPW-qcsr2} for an
overview of some of these). We remark that these are very special varieties: for
example, conical symplectic resolutions were shown recently to be
rigid (in fact, that there are finitely many of a given dimension and
bound on degree), in \cite{Nam-ftss}.

One of the fundamental tools in the study of complex Poisson (or
degenerate symplectic) varieties (not necessarily conical) is the
Darboux-Weinstein theorem (\cite{We}; see also \cite[Proposition
3.3]{Kalss}), which gives a local structure for such varieties.
Namely, recall that a symplectic leaf $Z$ of such a variety $X$ is
defined as a maximal connected locally closed subvariety on which the
tangent space $T_zZ$ is equal to the span of the Hamiltonian vector
fields at $z$, for all $z \in Z$.  Then the Darboux-Weinstein theorem
says that a formal neighborhood $\hat X_z$ of $z \in Z$ splits as a
product $\hat Z_z \times S$, for some formal transverse slice $S$ to
$Z$ at $z$.  This has been of fundamental use throughout the
literature, perhaps even more so recently, in understanding the
structure of the varieties and their quantization.  For example, one
can attach to the irreducible representations of the quantization
their support, which are the closures of symplectic leaves. This
allows one to apply geometry of $X$ to the representation theory of
its quantization. This idea perhaps appeared first in Lie theory,
where one can attach to each irreducible representation of a
semisimple Lie algebra a nilpotent (co)adjoint orbit, which is the
associated graded ideal of its kernel (i.e., the associated primitive
ideal) in the universal enveloping algebra, and study the primitive
ideals with fixed support. Losev showed in, e.g., \cite{Losqsa,
  Los-csra}, that, by quantizing the Darboux-Weinstein decomposition,
one can relate irreducible representations with a given support to
finite-dimensional representations of the quantization of the slice
$S$.

The main idea of the present paper is to replace a formal neighborhood
of a point $z \in X$ by a formal neighborhood of the punctured
line $\bC^\times \cdot z$, which allows one to generalize the
Darboux-Weinstein decomposition to a $\bC^\times$-equivariant one.

We say that $X$ admits a symplectic resolution if $X$ is normal and
there is a projective resolution of singularities $\tilde X \to X$
with $\tilde X$ symplectic.
% We will always assume (in the presence of a $\bC^\times$ action on
% $X$) that the resolution also admits a $\bC^\times$ action for which
% the map is $\bC^\times$-equivariant.
A significantly weaker condition is that $X$ be a symplectic
singularity \cite[Definition 1.1]{Beauss}: this requires that $X$ be
normal, symplectic on its smooth locus, and that the pullback of the
symplectic form under any (equivalently, every) resolution $\tilde X
\to X$ extend to a regular (but possibly degenerate) two-form on
$\tilde X$.

Let $\Delta^m := \Spf \bC[\![x_1,\ldots,x_m]\!]$ be the formal disk
(the reader preferring the analytic setting could instead work with a
small $m$-disk, modifying the statements accordingly).

\begin{theorem}\label{t:pre-main}
  Let $X$ be a complex variety with a $\bC^\times$-action and a homogeneous
  Poisson structure of degree $-k$.  Let $Y$ be a $\bC^\times$-stable
  leaf,\footnote{If $X$ is a union of finitely many symplectic leaves
    (e.g., under assumption (iii)), then this hypothesis is
    unnecessary, since homogenity of the Poisson structure implies
    that all symplectic leaves are all $\bC^\times$-stable.}  and $y
  \in Y$ a point with trivial stabilizer under $\bC^\times$.  Then:
\begin{enumerate}
\item[(i)] There is a decomposition of $\bC^\times$-formal
  schemes,
\begin{equation}\label{e:ehat}
\hat X_{\bC^\times \cdot y} \cong \hat Y_{\bC^\times \cdot y} 
\times S,
\end{equation}
for some formal scheme $S$ (equipped here with a trivial $\bC^\times$
action). 
\item[(ii)] $S$ is equipped with a canonical Poisson structure, and
  $\hat Y_{\bC^\times \cdot y} \cong \bC^\times \times
  \Delta^{\dim Y-1}$, equipped with a standard symplectic structure
  (see Theorem \ref{t:e-db} below). 
  % The projection to the first factor, $\hat X_{\bC^\times \cdot y}
  % \to \hat Y_{\bC^\times \cdot y}$, is Poisson.
\item[(iii)] If $X$ admits a symplectic resolution (or more generally
  is a symplectic singularity in the sense of \cite{Beauss}), then
  the isomorphism class of $\hat X_{\bC^\times \cdot y}$ as a formal
  $\bC^\times$-Poisson scheme is uniquely determined by the Poisson
  isomorphism class of $S$ and the dimension of $X$.
% (i.e., if $(X',Y',y',S')$ satisfy
%the same hypotheses as $(X,Y,y,S)$, $S' \cong S$ as formal Poisson schemes, and $\dim X = \dim X'$, then $\hat X'_{\bC^\times \cdot y'} \cong \hat X_{\bC^\times \cdot y}$ as formal $\bC^\times$-Poisson schemes).
\item[(iv)] If the condition in (iii) holds and also $S$ admits a
  $\bC^\times$ action giving its Poisson structure degree
  $-k$, then there is a $\bC^\times$-Poisson isomorphism of the form
  \eqref{e:ehat},
$\hat X_{\bC^\times \cdot y} \cong \hat
  Y_{\bC^\times \cdot y} \times S$ (i.e., the Poisson bivector
  on $\hat X_{\bC^\times \cdot y}$ is actually the sum of the
  canonical Poisson bivectors on $\hat Y_{\bC^\times \cdot y}$ and on
  $S$).\footnote{Note that this isomorphism differs slightly from
the one in (i), as required since the $\bC^\times$-action on
$S$ has changed; however, they become the same after including
either factor or after projecting to the first factor.}
 In this case, $S$ is Poisson isomorphic to every slice appearing
  in an ordinary Darboux-Weinstein decomposition, $\hat X_y \cong \hat
  Y_y \times S$.
\end{enumerate}
\end{theorem}
%Note that, under the hypotheses of (iii) (hence of (iv)) that $k$ is
%automatically positive, but this is not assumed in parts (i) and (ii).
%
The theorem is proved in \S \ref{ss:t-pre-main-pf}, using results from
\S$\!$\S \ref{s:sympl}--\ref{s:main}.  We also give a quantization of
the above result (Theorem \ref{t:quant}), where we show that every
$\bC^\times$-compatible quantization of $X$ admits a direct-product
decomposition parallel to the above.

We actually prove stronger (although more technical)
versions of the above theorem.  In particular, 
%it is not essential
%that the stabilizer of $y$ be trivial, as explained in Theorem
%\ref{t:main-sl-nf}.(ii), and 
the assumption of (iii) can be relaxed to requiring that every Poisson
vector field on $S$ ia Hamiltonian (the infinitesimal analogue of
the condition that every symplectic torus action is Hamiltonian), as
explained in Corollary \ref{c:ham-pbv-isom}. In Theorem
\ref{t:main-sr}, we explain why this latter condition is implied when
$X$ is a symplectic singularity, and hence when it admits a symplectic
resolution.  The assumption that the stabilizer of $y$ be trivial is
not strictly necessary, but is an important simplification; see
Corollary \ref{c:main-sl-nf} for a couple statements without this
assumption.

As we explain in \S \ref{s:main}, the hypotheses of (iii) and (iv)
arise naturally.  In particular, Examples \ref{ex:counterex1} and
\ref{ex:counterex2} are simple typical cases where assumption (iii)
does not hold and the conclusions of (iii) and (iv) fail.  Indeed, in
the more technical Theorem \ref{t:main-sl}.(iii) below, we will
describe the Poisson structure on $\hat X_{\bC^\times \cdot y}$
guaranteed in part (iii) above, which is a sum not merely of terms
coming from both $\hat Y_{\bC^\times \cdot y}$ and $S$, but a third
term, which in general represents an obstruction to writing $\hat
X_{\bC^\times \cdot y}$ as a product as in part (iv) above.
Nonetheless, assumptions (iii) and (iv) are not necessary for the
conclusions to hold: for example, the conclusion of (iv) holds in the
case $X=\mathfrak{sl}(2)^*$, as explained in the next subsection, even
though this case does not satisfy the assumptions of (iii) (hence
neither of (iv)).

On the other hand, we explain a conjecture that implies that the assumption
(iv) is unnecessary when $X$ is conical.  Namely, if $X$ is conical
and admits a symplectic resolution (or more generally is a symplectic
singularity),
% The hypotheses of the theorem imply in particular that the slice $S$
% appearing in an ordinary Darboux-Weinstein decomposition admits a
% $\bC^\times$ action. 
is a conjecture of Kaledin \cite[Conjecture
1.8]{Kal-gtsr} that, for 
% , for any conical $X$ admitting a
% $\bC^\times$-equivariant symplectic resolution (or more generally a
% conical symplectic singularity) and
every $x \in X$, the ordinary
Darboux-Weinstein slice $S$ at $x$ is conical (this holds in cases of
interest such as hypertoric and quiver varieties, Kostant-Slodowy
slices of the nilpotent cone, and linear quotient singularities).  The
theorem above motivates the following natural strengthening:
\begin{question}\label{q:slice}
  Suppose $X$ is conical and admits a $\bC^\times$-equivariant
  symplectic resolution with homogeneous symplectic form.
% of degree $k$.  
 Then for every point
  $x \in X$ with trivial stabilizer under $\bC^\times$:
\begin{enumerate}
\item Does the $S$ appearing in \eqref{e:ehat} admit a contracting
  $\bC^\times$-action?
\item Can the action be taken to make the Poisson structure on $S$ homogeneous?
\item Does this generalize to the situation where $X$ need not admit a
  symplectic resolution, but is a symplectic singularity (whose
  smooth locus has a homogeneous symplectic form)?
\end{enumerate}
\end{question}
Note that it is automatic in the conical case that
the degree of the generic symplectic form is positive (i.e., the
degree of the Poisson bracket is negative).
A positive answer to the question would imply that, for $X$ conical,
the assumption in Theorem \ref{t:pre-main}.(iv) holds automatically
(assuming still (iii)).   We will pose a quantum
version of the above in Question \ref{q:q-slice} below.
\begin{remark}\label{r:nam}
  In fact, as suggested by Namikawa, one can generalize the question
  to require only that $X$ be normal and conical and have a sympletic
  form on the smooth locus of positive degree, and not to require that
  $X$ be a symplectic singularity.  Moreover, as he pointed out, in
  this case, if the slices to all symplectic leaves are conical and
  have symplectic forms on their smooth locus of positive degree, then
  one can deduce \`a fortiori that $X$ is a symplectic singularity
  (by the proof of Lemma 2.4 of \cite{Nam-ess}).
\end{remark}

In this paper, we show that these questions have affirmative answers
for linear quotient singularities (which rarely admit symplectic
resolutions, as explained in \cite{Belscms,BSsra2} and references therein) and that (1) and
(2) have affirmative answers for hypertoric varieties, and we
explicitly compute the decompositions.  For the case where $X$ is the
nilpotent cone of a semisimple Lie algebra (or its Kostant-Slodowy
slices), see Remark \ref{r:walg} for a discussion.
\begin{remark}\label{r:weight-subtle}
  It might be tempting to ask a stronger question than (2) above: Can
  $S$ be taken to have a $\bC^\times$-action so that its Poisson
  structure has the same degree as that of $X$? However, the answer to
  this is negative in general.  For example (cf.~Remark \ref{r:walg}
  below), suppose $\mfg$ is a semisimple Lie algebra and equip
  $\mfg^*$ with its standard Poisson bracket of degree $-1$, defined
  by $\cO(\mfg^*)=\Sym \mfg$. Let $X \subseteq \mfg^*$ be the cone of
  elements $\langle x,-\rangle$ where $x \in \mfg$ is ad-nilpotent and
  $\langle-,-\rangle$ is the Killing form (i.e., the cone of elements
  whose coadjoint orbit is stable under dilation).  Then $X$ is a
  closed Poisson subvariety, so has a Poisson bracket of degree
  $-1$. Let $Y \subseteq X$ be any coadjoint orbit (which is
  automatically conical).  Then, a transverse slice to the orbit is
  given by the Kostant-Slodowy slice in $\mfg$ intersected with the
  nilpotent cone.  The latter has a well-known action, called the
  Kazhdan action, making it a Poisson cone with bracket, of degree
  $-2$.  The Kazhdan action admits a square root, giving the Poisson
  bracket degree $-1$, if and only if $e$ is even (i.e., for some
  $\mathfrak{sl}(2)$-triple $(e,h,f)$, then $\ad h$ has only even
  eigenvalues).  Note that $X$ does admit a $\bC^\times$-equivariant
  symplectic resolution, the well-known Springer resolution.

  When $e$ is not even, in general no $\bC^\times$ action exists
  giving the Poisson bracket degree $-1$: for example, in the case
  when $e$ is a subregular nilpotent and $\mfg=\mathfrak{sl}(n)$, then
  the slice is $\bC[x,y,z]/(xy+z^n)$, with Poisson bracket
  $\{x,y\}=nz^{n-1}, \{z,x\}= x, \{y,z\}=y$, and when $n$ is odd,
  there is no grading giving the bracket degree $-1$.
\end{remark}
We apply these techniques in particular to the cases of quotient
singularities (Section \ref{s:flq}) and hypertoric varieties (Section
\ref{s:hypertoric}).  In these cases, we give explicit
equivariant Darboux-Weinstein decompositions, which are Zariski local in the
hypertoric case (and sometimes in the quotient case, but more
generally \'etale local). These decompositions quantize to give tensor
product decompositions of the noncommutative deformations, which
unlike in previous literature now incorporate the $\bC^\times$-action:
this means one obtains isomorphisms of \emph{filtered} algebras, or
alternatively graded $\bC[\![\hbar]\!]$-algebras.

We mention one of the motivations for Theorem \ref{t:pre-main}: a weak
version, using only infinitesimal $\bC^\times$-actions (i.e., Euler
vector fields) appeared in the recent paper \cite{PS-pdrhhvnc}. In the
infinitesimal form, the result is immediate from the usual
Darboux-Weinstein theorem.
% : it says merely that, for an ordinary Darboux-Weinstein
% decomposition $\hat X_y \cong \hat Y_y \times S$, one can write
% the Euler vector field $\Eu_X$ of $X$ at $y$ as a sum, $\Eu_X|_{\hat
% Y_y} = \Eu_Y|_{\hat Y_y} + \eta$ of the Euler vector field on $Y$
% and a vector field $\eta$ on $\hat Y_y$ parallel to $S$.
In \cite{PS-pdrhhvnc}, this was employed to study the structure of a
canonical $\caD$-module on a Poisson variety \cite{ESdm,ESsym,ES-dmlv}
which represents invariants under Hamiltonian flow.  Namely, the latter
$\caD$-module admits a Jordan-H\"older decomposition whose composition
factors are local systems on the leaves, and in certain cases, using
Theorem \ref{t:pre-main} (or its infinitesimal version),
one can show that it is a direct sum of intermediate extensions of
explicit weakly equivariant local systems on these leaves
(\cite[Theorem 5.1]{PS-pdrhhvnc}).  In Corollary \ref{c:main-dmod} and
Remark \ref{r:main-dmod} below, we describe the local equivariant
structure of this $\caD$-module without requiring the hypotheses of
\cite[Theorem 5.1]{PS-pdrhhvnc} (which are only needed to express the
global structure as the aforementioned direct sum).

% Incorporating full $\bC^\times$-actions, as above, is significantly
% more subtle than merely infinitesimal actions, which in general do
% not integrate to $\bC^\times$-actions, and only require localization
% at a point $x$ instead of a punctured line $\bC^\times \cdot x$. In
% more detail, the subtlety has to do with the fact that $S \subseteq
% \hat X_y$ is only canonical up to Hamiltonian isotopy, which is
% incompatible with the $\bC^\times$-action. Although this was not
% needed for \cite[Theorem 5.1]{PS-pdrhhvnc}, in cases where the
% hypotheses there are not satisfied but our theorem applies (e.g.,
% when $X$ does not admit a symplectic resolution but only has
% symplectic singularities), the result here gives a description of
% the weakly equivariant structure of the $\caD$-module.

We begin the paper, in \S \ref{s:sympl}, with easier, but still to our
knowledge new, fundamental results on the Darboux theorem for
$\bC^\times$-equivariant symplectic structures on smooth varieties
with a nontrivial $\bC^\times$-action.  We show that, if we formally
localize along a punctured line $\bC^\times \cdot x$, the resulting
formal $\bC^\times$-Poisson scheme is completely classified by the
degree of the symplectic form and the dimension of the vector space,
and give the explicit formula for the structure. Then, using this
section, we proceed to our main results in \S \ref{s:main} and to
examples and applications in \S\!\S \ref{s:flq}--\ref{s:hypertoric}.

It should also be possible to give explicit formulas for the
corresponding decomposition for Slodowy slices, quiver varieties, and
more generally for Hamiltonian reductions of symplectic vector spaces
or varieties (as well as to answer Question \ref{q:slice} above in these
cases).
% which we will discuss elsewhere. 
The quantizations of these
would yield decompositions for $U\mfg$ and more generally for
$W$-algebras, as well as for quantized quiver varieties and quantum
Hamiltonian reductions.

\subsection{The case of $\mathfrak{sl}(2)$ and semisimple
 Lie algebras}\label{ss:ss-Lie}
Let us illustrate our decomposition result in a simple
 example: $X = \mathfrak{g}^*$ for
$\mathfrak{g}=\mathfrak{sl}_2(\bC)$.  
This is equipped with a standard
Poisson structure, which is given by the Lie bracket: the bracket on
$\cO(X)=\Sym \mathfrak{g}$ is the unique extension of the Lie bracket
on $\mathfrak{g}$ satisfying the Leibniz rule,
$\{fg,h\}=g\{f,h\}+f\{g,h\}$. The symplectic leaves are the coadjoint
orbits under $G = \SL_2(\bC)$ (which are equipped with the symplectic
forms $\omega_{G\cdot \phi}(a(x)|_\phi, a(y)|_\phi) = \phi([x,y])$,
for $\phi \in \mfg^*$, $x,y \in \mfg$, and $a(x)$ and $a(y)$ the
vector fields of the infinitesimal adjoint action of $\mfg$ on
$\mfg^*$).  

Although this does not satisfy the hypotheses of Theorem
\ref{t:pre-main}.(iii), we can still give its decomposition (e.g.,
Theorem \ref{t:main-sl} below still applies).

We are interested in decompositions along $\bC^\times$-stable
symplectic leaves other than the vertex, and the only such leaf is the
unique nontrivial nilpotent coadjoint orbit, $G \cdot \chi$, for $\chi = \langle
e, -\rangle$ (with $\langle -,-\rangle$ the Killing form). To obtain a
$\bC^\times$-equivariant decomposition, we can formally localize along
the punctured line $\bC^\times \cdot \chi$.  Let $c = 2ef +
\frac{1}{2}h^2$ be the Poisson central element, which has degree
two. Then we have
\begin{multline}
\widehat{\cO(X)}_{\bC^\times \cdot \chi} = \bC[f,f^{-1}][\![e,h]\!]
\cong \bC[f,f^{-1}][\![h]\!] \hat \otimes \bC[\![c]\!]
\\ \cong \cO(\widehat{\bA^2}_{\bC^\times \cdot \chi} \times (\widehat{\bA^1})_0)
\cong \cO(\widehat{G \cdot \chi}_{\bC^\times \cdot \chi} \times (\widehat{\bA^1})_0),
\end{multline}
which is a $\bC^\times$-equivariant Poisson decomposition. 
\begin{remark} 
  Note the subtlety that, to get a $\bC^\times$-Poisson direct product
  decomposition, the slice $\bC[\![c]\!]$ had to be generated in
  degree two.
  % In this case, since the Poisson structure is zero on this slice,
  % one does not need to worry about the degree of its Poisson bracket
  % (the condition appearing in Theorem \ref{t:pre-main}.(iv),
  % although condition (iii) is not satisfied in this case).  In the
  % examples in this paper, we will generally consider the case where
  % $X$ has
  % finitely many leaves (or even admits a symplectic resolution);
  % this would mean working with the nilpotent cone instead of $X$
  % (i.e., modding by the central element $c$), and then we get the
  % trivial decomposition $\cO(G \cdot \chi) \cong \cO(G \cdot \chi)
  % \otimes \bC$.
\end{remark}
In fact, in this case we have
a much stronger statement: the decomposition above holds Zariski locally:
\begin{equation}\label{e:sl2}
  \cO(X\setminus\{f=0\}) = \bC[f,f^{-1},h] \otimes \bC[c] \cong
  \cO((\bC^\times \times \bA^1) \times \bA^1),
\end{equation}
which is already a $\bC^\times$-Poisson isomorphism. 

To obtain a filtered quantization, we first give a filtered
quantization of $\cO(X)[f^{-1}]$.  This is given by the Ore
localization $U\mfg[f^{-1}]$, once we verify that $S=\{f^m \mid m \geq
1\}$ is a right denominator set (in the terminology of \cite[\S
10]{Lam-lmr}, which also explains why this is sufficient to obtain an
Ore localization).  First, we need to demonstrate the right Ore
condition: for all $a \in U\mfg$ and all $f^i \in S$, there exists $b
\in U\mfg$ and $f^j \in S$ satisfying $af^j = f^ib$.  It suffices to
prove this when $a$ ranges over algebra generators of $U\mfg$, so for
$a \in \mfg$; in this case, setting $j=i+1$, we have $a f^{i+1} = f^i
(fa + [a,f])$, as desired.  Next, we need to verify the right
reversibility condition: if $a \in U\mfg$ satisfies $f^ia=0$ for some
$f^i \in S$, then there exists $f^j \in S$ with $af^j = 0$. This is
satisfied trivially since $U\mfg$ is a domain (it has no
zerodivisors).

Now, let $C := ef+fe+\frac{1}{2}h^2 \in U\mfg$ be the Casimir element. Then we obtain the decomposition:
\begin{equation}\label{e:sl2-q}
  U \mfg[f^{-1}] = \bC\langle f,f^{-1},h \rangle/([h,f]+2f) \otimes \bC[C]
  = \bC \langle x,x^{-1},y \rangle/([x,y]-x) \otimes \bC[C],
\end{equation}
with $x=f$ and $y = \frac{1}{2} h$, with $x$ and $y$ in filtered
degree one and $C$ in filtered degree two.  For a graded deformation
quantization, letting $U_\hbar \mfg := T\mfg[\![\hbar]\!] /
(xy-yx-\hbar[x,y])_{x,y \in \mfg}$, we get
\begin{equation}\label{e:sl2-qh}
U_\hbar \mfg[f^{-1}] = \bC \langle x,y \rangle[\![\hbar]\!]/([x,y]-\hbar x) 
\hat \otimes \bC[C],
\end{equation}
which is graded with $|\hbar|=|x|=|y|=1$ and $|C|=2$.  We can also
invert $\hbar$ and get a decomposition over the Laurent field
$\bC(\!(\hbar)\!)$.  

If we are interested in quantizations of the nilpotent cone, $\Spec
\cO(X)/(c)$, we can divide \eqref{e:sl2-q} by the ideal
$(C-\lambda)$ (or \eqref{e:sl2-qh} by $(C-\lambda \hbar^2)$) for
$\lambda \in \bC$, and we recover the fact that inverting $f$ in every
quantization $U\mfg/(C-\lambda)$ yields the algebra of differential
operators on $\bC^\times$.
\begin{remark}\label{r:walg}
  We believe that the above can be generalized to arbitrary
  semisimple $\mathfrak{g}$ in the following way.
  We are interested in the Poisson variety $\mfg^*$, under dilation
  action, which gives the Poisson bracket degree $-1$. The symplectic
  leaves are the coadjoint orbits. We consider such an orbit closed
  under the $\bC^\times$-action, say $G \cdot \chi$ for $\chi \in
  \mfg^*$.  Then there is a standard construction of a transverse
  slice to this orbit: For $\langle -, - \rangle$ the Killing form,
  let $e \in \mfg$ be such that $\langle e,x\rangle=\chi(x)$ for all
  $x \in \mfg$; then $e$ is ad-nilpotent. The Jacobson-Morozov theorem
  guarantees the existence of $h,f\in \mfg$ such that $(e,h,f)$
  generate a subalgebra of $\mathfrak{sl}_2$. To this is associated a
  transverse slice $S := \chi + \ker(\ad^* f)$, with $(\ad^* x)(\phi)
  := \phi \circ \ad (-x)$.  The tangent space to the orbit $G \cdot
  \chi$ can be described as $V^*$, for $V:=[f,\mfg] \subseteq \mfg$,
  equipped with the symplectic form $\omega_V(x,y)=\chi([x,y])$.

  There is a canonical $\bC^\times$ action on $\mfg^*$ which preserves
  $S$ and restricts there to a contracting action to $\chi$, called
  the Kazhdan action, given by $\lambda \cdot \phi = \lambda^{-2}
  \lambda^{\ad^* h}(\phi)$ for $\lambda \in \bC^\times$ and $\phi \in
  S$. However, this gives the Poisson bivector degree $-2$, unlike the
  standard dilation action on $\mfg^*$ above.  To fix this, we assume
  that $e$ is even, which means that $\ad(h)$ acts only with even
  eigenvalues.  Then, the Kazhdan action admits a square root,
  $\lambda \mapsto \lambda^{-1} \lambda^{\frac{1}{2} \ad^* h}$.  Using
  this action, the condition of Theorem \ref{t:pre-main}.(iv) on the
  degree of the Poisson bivector is satisfied.\footnote{Since $S$ is
    an affine space, even if $e$ is not even, we can still pick some
    $\bC^\times$-action for which the Poisson structure on $S$ has
    degree $-1$, but this is not natural. Moreover, as explained in
    Remark \ref{r:weight-subtle}, the intersection $S \cap
    \text{Nil}(\mfg^*)$ with the nilpotent cone does \emph{not} in
    general have a $\bC^\times$ action giving the bracket degree
    $-1$.}
  % (one can probably relax this condition by passing to a two-fold
  % cover of $\mfg^*$, adjoining $\sqrt{f}$ to $\cO(\mfg^*)$).

  Then, we have a $\bC^\times$-equivariant isomorphism (a
  $\bC^\times$-equivariant, quasiclassical analogue of \cite[Theorem
  1.2.1]{Losqsa}):
\begin{equation}
  \widehat{\mfg^*}_{\bC^\times \cdot \chi} \cong \widehat{V^*}_{\bC^\times \cdot \chi} \times
  \hat S.
\end{equation}
Following \cite{Losqsa}, this yields a decomposition of a localization
of a certain completion of the enveloping algebra $U_\hbar \mfg$.  Let
$\mfg(i)\subseteq \mfg$ denote the $i$-weight space of $\ad h$. Equip
$\mfg(-1)$ with the symplectic form $(x,y) = \chi([x,y])$, and let
$\mathfrak{l} \subseteq \mfg(-1)$ be a Lagrangian.  Set $\mathfrak{m}
:= \bigoplus_{i \leq -2} \mfg(i) \oplus \mathfrak{l}$, and let
$\mathfrak{n} := \mathfrak{m} \cap \ker(\chi)$. 
% Assume $\mfg$ is even, so that $\ad h$ acts only with even
% eigenvalues (this assumption can perhaps be relaxed at the cost of
% passing to a $2$-fold cover, i.e., adjoining $\sqrt{f}$).
Then \cite[Theorem
1.2.1]{Losqsa} should strengthen to the following graded filtered
isomorphism:
\begin{equation}\label{e:ug-decomp}
  \widehat{U \mfg}_{\mathfrak{n}}[f^{-1}] \cong \caD(\bC^\times \times \Delta^{\dim Y/2 - 1}) \hat \otimes \mathcal{W}_\chi,
\end{equation} 
which is compatible with the $\bC^\times$-action by $\lambda \mapsto
\lambda^{\ad h/2}$ on the left-hand side, and the action on the first
factor on the right-hand side by dilating in the $\bC^\times$
direction. 
Here, $\mathcal{W}_{\chi}$ is the $W$-algebra quantizing $S$, which
was denoted $U(\mfg,e)$ in \cite{Losqsa}, and is defined as $(U\mfg /
U\mfg \cdot \mathfrak{m}')^{\mathfrak{m}}$ for $\mathfrak{m}' =
\{x-\chi(x) \mid x \in \mathfrak{m}\} \subseteq U\mfg$. We equip it
with the filtration obtained by reducing by half the degrees of the
Kazhdan filtration (compatible with our grading on $S$ above).  From
this, \cite[Theorem 1.2.1]{Losqsa} follows (in the case of even $e$)
by completing along $\mathfrak{m}'$. This also recovers the example
above \eqref{e:sl2},\eqref{e:sl2-q}, in the case
$\mathfrak{g}=\mathfrak{sl}(2)$.

When we replace $X$ with the nilpotent cone $X = \text{Nil}(\mfg^*)$,
we obtain a decomposition $\hat X_{\bC^\times \cdot \chi} \cong
\widehat{V^*}_{\bC^\times \cdot \chi} \times (\widehat{S \cap
  X})$. In this case, the hypotheses of Theorem \ref{t:pre-main} are
satisfied; the content here is the explicit identification of the
slice with the Kostant-Slodowy slice $\widehat{S \cap X}$. The
quantization then is the quotient of \eqref{e:ug-decomp} by the
augmentation ideal of the center of $U\mfg$. As a corollary, we can
deduce that Question \ref{q:slice} has a positive answer for the
nilpotent cone.
\end{remark}

\subsection{Conventions}\label{ss:conv}
We will work with varieties or formal schemes over $\bC$ (so
$\bA^n=\bC^n$).  When we take a product of formal schemes, we always
mean the formal scheme obtained by completing the corresponding tensor
product of rings of functions, i.e., $\Spf A \times \Spf B = \Spf (A
\hat \otimes B)$, where if $A$ is given the $I$-adic topology and $B$
the $J$-adic topology (for $I \subseteq A$ and $J \subseteq B$ ideals)
then $A \hat \otimes B$ is the completion of $A \otimes B$ in the
$(I+J)$-adic topology.  When we say a $\bC^\times$-variety (or formal
scheme), we mean a variety (or formal scheme) equipped with a
$\bC^\times$-action. For an affine variety, this just means that the
algebra of functions is $\bZ$-graded.  When we say a
$\bC^\times$-Poisson variety (or formal scheme), we mean one equipped
with a $\bC^\times$-action and a Poisson structure homogeneous for
this action.
% A conical variety (or formal scheme) is one equipped
% with a $\bC^\times$-action contracting to a point (so in the variety
% case this means it is affine with a nonnegatively graded algebra of
% functions, with $\bC$ in degree zero).  
A $\bC^\times$-Poisson (iso)morphism is a $\bC^\times$-equivariant
Poisson (iso)morphism.  When $f$ is a homogeneous element of a
$\bZ$-graded algebra, then $|f|$ will denote its degree.

When we say ``symplectic leaf,'' $Y$ of a Poisson variety $X$, we will
always mean a algebraic symplectic leaf, i.e., $Y$ is a locally closed
algebraic subvariety such that, for every $y \in Y$, the tangent space
$T_y Y$ is the span of the restriction of all Hamiltonian vector
fields (originally defined in any neighborhood of $y$) to $Y$.
Moreover, we always assume symplectic leaves are connected and that
they are maximal (i.e., they are not proper open subsets of a larger
locally closed connected subvariety with the property in the previous
sentence). (Note that, since we are usually only concerned with local
properties, these last two conditions of maximality and connectedness
will be irrelevant for most of our results.)  Here is an important
example: if $X$ is a union of finitely
many algebraic symplectic leaves (which holds, for instance, when $X$ admits a
symplectic resolution or is a symplectic singularity) then the
symplectic leaves are the connected components of the loci $X_i
\subseteq X$ of points $x \in X$ where the restriction of Hamiltonian
vector fields span a subspace of dimension $i$ in $T_x X$.
% We remark
% that, in this case, $X$ can always be written as a union of at most
% $\dim X + 1$ leaves, by taking for each $m \leq \dim X$ the leaf whose
% points $x$ are those where the span of the restriction of Hamiltonian
% vector fields to $x$ has dimension $m$.

\subsection{Acknowledgements}
I would like to thank Ivan Losev and Hiraku Nakajima for useful
discussions, and Yoshinori Namikawa for his feedback on Question
\ref{q:slice}. I am grateful to the anonymous referees for important
suggestions and corrections. The impetus for this work was
\cite{PS-pdrhhvnc}, and I thank Nick Proudfoot for his collaboration
on that project.  This work was partially supported by NSF grant
DMS-1406553.

\section{$\bC^\times$-equivariant formal symplectic
  geometry}\label{s:sympl}
Given an affine $\bC^\times$-variety $X$ and $x \in X$, let $\bar x$
denote the image of $x$ in the categorical quotient $X /\!/ \bC^\times =
\Spec \cO(X)^{\bC^\times}$.
\begin{lemma}\label{l:prod}Let $X$ be an irreducible affine 
variety with a faithful
$\bC^\times$-action, and $x \in X$ a point with trivial stabilizer. 
Then there is a $\bC^\times$-stable affine open subvariety
$U$ containing $x$ together with 
an isomorphism $U \cong \bC^\times \times U/\!/\bC^\times$, 
such that $x \mapsto (1,\bar x)$.
\end{lemma}
\begin{remark}
  By Sumihiro's theorem, we can drop the assumption that $X$ is affine
  if we assume that it is normal, since then every orbit is contained
  in a $\bC^\times$-stable open affine subvariety.  
  % Note: for a counterexample without the normality assumption,
  % we could take the product of a projective nodal
  % curve with C*, acting on C* by the standard action and on the
  % projective nodal curve by the standard action, for which every
  % C*-stable neighborhood of the singular line C* is the whole
  % non-affine curve.
\end{remark}
\begin{proof}[Proof of Lemma \ref{l:prod}]
  Let $t \in \cO(X)$ be any homogeneous function (of weight one) which
  restricts on the line $\bC^\times \cdot x$ to a homogeneous
  coordinate function (of weight one by our assumption).  Let $U$ be
  the complement of the locus where $t=0$. Then $U$ is stable under
  the $\bC^\times$-action, so $\cO(U)$ is spanned by homogeneous
  functions.  Every homogeneous function is of the form $f t^k$ where
  $f$ has weight zero, i.e., $f \in \cO(U)^{\bC^\times}$, and $k$ is an
  integer.  Thus the inclusion of algebras $\cO(U)^{\bC^\times} \otimes
  \bC[t,t^{-1}] \to \cO(U)$ is an isomorphism, i.e., $U \cong
  \bC^\times \times U/\!/\bC^\times$ as $\bC^\times$-varieties (giving
  $U/\!/\bC^\times$ the trivial $\bC^\times$-action).
\end{proof}
From now on, we will use the following notation for a locally closed
affine subvariety $Y$ of a (not necessarily affine) 
variety $X$.  Let $U \subseteq X$
be an open affine subset such that $Y$ is closed in $U$ (i.e.,
obtained by inverting an element whose vanishing locus on $\bar Y$ is
$\bar Y \setminus Y$). Then the completion $\hat \cO(X)_Y$ is defined
as the completion $\hat \cO(U)_Y$, which clearly does not depend on
the choice of open affine subset $U$. We set $\hat X_Y := \Spf \hat
\cO(X)_Y$. We deduce the following corollary:
\begin{corollary}\label{c:fn}
  Let $X$ be a (not necessarily affine) irreducible variety with a
  faithful $\bC^\times$-action, and $x \in X$ a point having trivial
  stabilizer.  Then the formal neighborhood $\hat X_{\bC^\times \cdot
    x}$ of $\bC^\times \cdot x$ is $\bC^\times$-equivariantly
  isomorphic to the product $\bC^\times \times Z$,
  where $Z=\Spf \cO(Z)$ is a formal affine scheme with the trivial
  $\bC^\times$-action.
\end{corollary}
Note here that $\bC^\times \times Z$ is, by definition, $\Spf
\bC[t,t^{-1}] \hat \otimes \cO(Z)$ (cf.~\S \ref{ss:conv}). In this
case, $\cO(Z)$ is equipped with the $\mathfrak{p}$-adic topology for
$\mathfrak{p} \subseteq \hat \cO(Z)$ a maximal ideal, so the completed
tensor product is with respect to the $(\mathfrak{p})$-adic topology.
\begin{proof}[Proof of Corollary \ref{c:fn}]
  The only thing that has to be explained is how to remove the
  affineness assumption.  The point is that $\bC^\times \cdot x$ is
  still affine, so the completion $\hat X_{\bC^\times \cdot x} = \Spf
  \hat \cO(X)_{\bC^\times \cdot x}$ is still an affine formal scheme.
% (indeed, if we take a (not necessarily
%  $\bC^\times$-stable) open neighborhood $U$ of $\bC^\times \cdot x$,
%  with $\mathfrak{p}$ the ideal of $\bC^\times \cdot x$, then $\cO(U)
%  / \mathfrak{p}^m$ is independent of the choice of $U$).  
  Now, as $\hat \cO(X)_{\bC^\times \cdot x}$ is the completion
  of a ring with a locally finite $\bC^\times$-action, it
  it is topologically (in the $\mathfrak{p}$-adic topology, with
  $\mathfrak{p}$ the ideal of $\bC^\times \cdot x$ as before) spanned
  by homogeneous elements.  The same proof as before applies to show
  that the inclusion of algebras $\bC[t,t^{-1}] \hat \otimes \hat
  \cO(X)_{\bC^\times \cdot x}^{\bC^\times} \to \hat \cO(X)_{\bC^\times
    \cdot x}$ is an isomorphism.
\end{proof}

Now, the Darboux theorem for formal neighborhoods goes through in this
context. Let $\Delta$ denote the formal polydisc, i.e., the formal
neighborhood of the origin in $\bA^1$. For all $m \geq 1$, let
$\Delta^m$ denote the formal $m$-polydisc, i.e., $\Spf
\bC[\![z_1,\ldots,z_m]\!]$.
%  (this is a slight abuse of notation, since
% this is the completed product of $\Delta$ with itself $m$ times)
% under the presence of a $\bC^\times$-action with homogeneous
% symplectic form.
Let us call the standard symplectic structure of degree $k$ on
$\bC^{\times} \times \Delta^{2n-1}$, coordinatized as $\Spf
\bC[t,t^{-1}][\![u,z_1,\ldots,z_{2n-2}]\!]$, with $u,z_i$ in degree
zero, the following:
\begin{equation}\label{e:ss-deg-k}
t^{k-1}dt \wedge du + \sum_{i=1}^{n-1} d(t^kz_{2i-1}) \wedge d z_{2i}.
\end{equation}
% Note that, even though $t^{k/2}$ is not in $\bC[t,t^{-1}]$ when $k$
% is odd, we could have put in $t^{k/2}z_i$ for all $i$ above,
% producing a well-defined one-form on $\bC^{\times} \times
% \Delta^{2n+1}$ since the half-integers cancel.
Let us also record the Poisson bivector in the above situation.  We
make the substitution $z'_{2i-1} := t^k z_{2i-1}$ for all $1 \leq i
\leq 2n-1$.
\begin{equation}\label{e:ss-deg-k-pb}
\pi_{2n,-k} := t^{-k+1}\partial_u \wedge \partial_t + \sum_{i=1}^{n-1} \partial_{z_{2i-1}'} \wedge \partial_{z_{2i}}.
\end{equation}.
\begin{theorem}\label{t:e-db}
  Any symplectic structure on $\bC^\times \times \Delta^{2n-1}$ of
  degree $k$ can be taken to the standard one by a
  $\bC^\times$-equivariant change of coordinates.
\end{theorem}
\begin{remark}
  As pointed out by a referee, at least in the case $k \neq 0$,
  there is a simple interpretation and proof of this statement using
  contact geometry: a symplectic structure $\omega$ on $\bC^\times
  \times \Delta^{2n-1}$ of degree $k \neq 0$ is of the form $\omega =
  d p^*(\theta)$ where $p: \bC^\times \times \Delta^{2n-1} \to
  \Delta^{2n-1}$ is the projection and $\theta \in
  \Omega^1(\Delta^{2n-1}) \otimes O(k)$ is a contact structure on
  $\Delta^{2n-1}$ equipped with the trivial bundle $O(k)$ viewed as
  $\bC^\times$-equivariant with the weight $k$ action of
  $\bC^\times$. Then $\theta$ can always be put in the standard form
  $\theta = (du + \sum_i z_{2i-1} dz_{2i}) \otimes t^k$.  (In other
  words, we can write $\omega = d(t^k p^* \alpha)$ where $\alpha \in
  \Omega^1(\Delta^{2n-1})$ is a usual contact one-form on
  $\Delta^{2n-1}$.)  Note that, when $k=0$, then $d(p^* \theta) = p^*
  d\theta$ which is degenerate for any one-form $\theta \in
  \Omega^1(\Delta^{2n-1})$, so that this proof does not seem to work.
\end{remark} 
\begin{proof}
  As before, write $\bC^\times \times \Delta^{2n-1} = \Spf
  \bC[t,t^{-1}][\![u,z_1,\ldots,z_{2n-2}]\!]$, and choose coordinates
  $u,z_1,\ldots,z_{2n-2}$ so that the restriction of the symplectic
  structure $\omega$ to $t=1, u=z_1,\ldots,z_{2n-2}=0$ is the standard
  one,
\[
\omega|_{(1,0,\ldots,0)} = dt \wedge du + \sum_{i=1}^{n-1} dz_{2i-1}
\wedge dz_{2i}.
\]
Then, since it is homogeneous, the symplectic structure must have the
form
\[
\omega = t^{k-1} dt \wedge du + \sum_{i=1}^{n-1} d(t^kz_{2i-1}) \wedge
d z_{2i} + \omega',
\]
where $\omega'$ is a closed two-form vanishing at the ideal
$(t-1,u,z_1, \ldots, z_{2n-2})$. Since $\omega'$ is homogeneous, it
must vanish at the entire locus $t\neq 0, u=z_1=\cdots=z_{2n-2}=0$,
and hence at the ideal $(u,z_1, \ldots, z_{2n-2})$.

Now we are in a position to apply Moser's trick, as in the proof of
the usual Darboux theorem. First, note that $\omega_c := \omega - c
\omega'$ is closed and nondegenerate for all $c \in \bC$. We need to
show that there exists a $\bC^\times$-equivariant symplectomorphism
$\Phi$ such that $\Phi^* \omega_1 = \omega_0$.  To do so, first note
that all closed two-forms on $\bC^\times \times \Delta^{2n-1}$ are
exact, so we can write $\omega' = d \alpha$ for some one-form
$\alpha$, also of weight $k$ with respect to the $\bC^\times$
action. Then, we can consider for all $0 \leq c \leq 1$ the vector
field $\theta_c$ such that $i_{\theta_c}(\omega_c) = -\alpha$.  Since
$\alpha$ and $\omega_c$ have weight $k$, $\theta_c$ will have weight
zero.  It moreover vanishes on $(u,z_1,\ldots,z_{2n-2})$.  As a
consequence, it integrates in $\bC^\times \times \Delta$ to an
automorphism $\Phi_s$ such that $\Phi_0 =\Id$ and $\frac{d}{ds}
\Phi_s|_{s=c} = \theta_c$.  Then the calculation in Moser's trick
shows that $\Phi_c^* \omega_c = \omega_0$ for all $c$, and in
particular that $\Phi_1^* \omega_1 = \omega_0$ as desired. Namely,
\[
\frac{d}{ds} (\Phi_s^* \omega_s) = \Phi_s^* L_{\theta_s} \omega_s +
\Phi_s^* \omega' = \Phi_s^* (d i_{\theta_s} + i_{\theta_s} d)\omega_s
+ \Phi_s^* \omega' = -\Phi_s^* d\alpha + \Phi_s^* \omega' = 0,
\]
so $\Phi_c^* \omega_c$ is constant, and $\Phi_0=\Id$ implies
$\Phi_0^*\omega_0=\omega_0$.
% consider the Poisson bracket $\{-,-\}'$ of the form
% \[
% \{t,u\}' = t^{k-1}, \quad \{z_{2i-1},z_{2i}\}' = t^k, 
% \]
% with all other brackets zero.  The bracket $\{-,-\}_\omega$ associated
% to $\omega$ has this form 
% up to elements in the ideal $(u,z_1,\ldots,z_{2n-2})$. We need to show that,
% by coordinate changes which are the identity modulo $(u,z_1,\ldots,z_{2n-2})$,
% we can make $\{-,-\}_\omega$ equal $\{-,-\}'$.  It suffices inductively to
% show that, if 
% $x\in \{t,u,z_1,\ldots,z_{2n-2}\}$ and $y \in \{z_1,\ldots,z_{2n-2}\}$
% are distinct coordinate functions
% and
% the error
% $\epsilon_{x,y} := \{x,y\}_\omega - \{x,y\}'$ is in $(u,z_1,\ldots,z_{2n-2})^m$,
% then we can make a coordinate change which is the identity modulo
% $(u,z_1,\ldots,z_{2n-2})^{m+1}$ after which $\epsilon_{x,y} \in (u,z_1,\ldots,z_{2n-2})^{m+1}$ and all other error terms do not change modulo $(u,z_1,\ldots,z_{2n-2})^{m+1}$. If $y=z_{2j}$, then the change is $x \mapsto x-t^{-k}z_{2j-1} \epsilon_{x,y}$
% (and similarly if $y=z_{2j-1}$ then we make the change $x \mapsto x+t^{-k} z_{2j} \epsilon_{x,y}$).
\end{proof}
% \begin{remark}
% following Remark \ref{r:darboux}. Namely, there the
%   procedure is shown how to apply an automorphism after which
%   $\{u,z_i\}=0$ for all $i$.  This automorphism will by construction
%   be $\bC^\times$-equivariant, and the same argument can be applied
%   with $u$ replaced by $t, z_1, z_2$, and so forth, so that we will
%   get $\{t,z_i'\}= 0$ for all $i$, $\{z_{2i-1}',z_{2i}'\}=1$ for all
%   $i$, and other brackets among the $z_i'$ zero. Then $\{t,u\}$ will
%   have zero bracket with all the $z_{i}'$ and have weight $k-1$, hence
%   $\{t,u\} = t^{k-1} f$ 
% \end{remark}
Putting the two results together, we obtain the following. 
\begin{corollary}\label{c:ss-nf}
  Let $X$ be a smooth symplectic variety with a faithful $\bC^\times$-action
  and a homogeneous symplectic structure of degree $k$.  Then for any
  $x \in X$ which has trivial stabilizer, a formal neighborhood of
  $\bC^\times \cdot x$ is isomorphic to $\bC^{\times} \times
  \Delta^{\dim X - 1}$ with the standard symplectic structure
  \eqref{e:ss-deg-k}.
\end{corollary}
We can also deduce the classification of symplectic structures in the
case that the $\bC^\times$ action is not faithful:
\begin{corollary}\label{c:ss-nf-nf}
  Equip $\bC^\times \times \Delta^{2n-1}$ with a $\bC^\times$ action
  which along the line $\bC^\times \times \{0\}$ has the form $\lambda
  \cdot (t,0) = (\lambda^\ell t,0)$.  Then any weight $k$ symplectic
  form can be written in some choice of homogeneous coordinates
  $(t,u,z_1,\ldots,z_{2n-2})$ (with $t$ invertible of weight $\ell$
  and restricting to the standard coordinate on $\bC^\times \times
  \{0\}$) in the standard form,
\begin{equation}\label{e:ss-deg-k-nf}
dt \wedge du + \sum_{i=1}^{n-1} dz_{2i-1} \wedge d z_{2i}.
\end{equation}
The same holds for the completion of a smooth symplectic variety with
a $\bC^\times$ action along a line $\bC^\times \cdot x$ where the stabilizer
of $x$ is the group of $\ell$-th roots of unity.
\end{corollary}
Note that there is no need to multiply the first term by a power of $t$ since now we can take $u$ to have nonzero degree (in fact, degree $k-\ell$).
\begin{proof}
  The proof is similar to that of Theorem \ref{t:e-db} and Lemma
  \ref{l:prod}.  First, we can pick $t$ to be a weight $\ell$ element
  mapping modulo the ideal, call it $\tilde J$, of $\bC^\times$ to the
  coordinate $t$.  By linear algebra, a homogenous symplectic form on
  a graded vector space can be written in standard form in terms of a
  homogeneous basis, and we can take this to include any nonzero fixed
  homogeneous element.  Thus, at $1 \in \bC^\times$, we can find
  homogeneous cotangent vectors $\bar u,\bar z_1,\ldots,\bar z_{2n-2}$
  so that $\omega = \bar t \wedge \bar u + \cdots + \bar z_{2n-1}
  \wedge \bar z_{2n-2}$. Lifting these to homogeneous generators of
  the ideal $(u,z_1,\ldots,z_{2n-2})$, we obtain homogeneous
  coordinates $t,u,z_1,\ldots,z_{2n-2}$ so that, at $1 \in
  \bC^\times$, $\omega$ restricts to $dt \wedge du + dz_1 \wedge dz_2
  + \cdots + dz_{2n-3} \wedge dz_{2n-2}$.  The rest of the argument is
  the application of Moser's trick as in the proof of Theorem
  \ref{t:e-db}.
\end{proof}
\begin{remark}\label{r:efp}
  We are interested here in the case where $x$ is not a fixed point
  since we are going to study symplectic leaves of
  $\bC^\times$-Poisson varieties (with homogeneous Poisson bracket of
  some degree) which are stable under the $\bC^{\times}$-action, but
  on which this action is nontrivial (the main example being Poisson
  cones with finitely many symplectic leaves, where all leaves other
  than the vertex have this form).  

  In the case of fixed points, the situation is different.  In the
  case that $x$ is a fixed point such that, in some neighborhood $U$ of $X$,
 the limit $t \cdot y$ exists as $t \to 0$ for all $y \in U$,
% (or equivalently
%  as $t \to \infty$), 
  which is called an \emph{elliptic fixed point}
  (as studied recently in, e.g., \cite{BDMN-ccdqsav}), the
  Bialynicki-Birula decomposition theorem yields an analogue of
  Corollary \ref{c:fn}, i.e., that in a formal neighborhood of $x$, we
  have an equivariant isomorphism $\hat X_x \cong \Delta^{2n} =
  \Spf[\![x_1, \ldots, x_{2n}]\!]$ where each $x_i$ is homogeneous of
  some nonnegative degree. Then, the formal Darboux theorem applies
  with the same proof, where we only require equivariant changes of
  coordinates, so that there is an equivariant change of coordinates
  under which the symplectic structure is
\[
\sum_{i=1}^n dx_{2i-1} \wedge dx_{2i},
\]
of degree $|x_{2i-1}|+|x_{2i}|$ (which therefore cannot depend on
$i$).
\end{remark}
\subsection{Proof without Moser's trick}\label{ss:without-moser}
It is instructive to give a proof of Theorem \ref{t:e-db} using only
Poisson brackets and not symplectic forms, and hence without using
Moser's trick, since these will be unavailable in the singular Poisson
case we consider in the sequel.  Thus, the reader may wish to skip
this subsection for now, and return to it before reading the proof of
Theorem \ref{t:main-sl} given in \S \ref{ss:t-main-sl-pf} below, which
uses the same ideas.  In two places, we give here a simpler, more
explicit argument than the one we need in \S \ref{ss:t-main-sl-pf},
thereby avoiding the use of \cite[Lemma 3.2]{Kalss}, although we also
outline below how the latter argument would go.

Define the ideals $J := (z_1,\ldots,z_{2n-2})$ and $\tilde
J:=(J,u)$. Define the notation $z_{2i-1}' = t^k z_{2i-1}$ and
$z_{2i}'=z_{2i}$. Then, we need to perform a coordinate change which
is the identity modulo $\tilde J$ so that
\[
\{t,u\}=t^{1-k}, \quad \{z_{2i-1}', z_{2i}'\} = 1,
\]
and all other brackets of the $t,u,z_1',\ldots,z_{2n}'$ zero.

The first step is to change the $z_i'$ so that $\{t,z_i'\}=0$ for all
$i$. As a result of this, the degree-zero part of the centralizer of
$t$ will be $\bC[\![z_1,\ldots,z_{2n-2}]\!]$.  We will do this
explicitly in the next paragraph.  We would like to point out, though,
that this computation can be replaced with the following more conceptual
argument: for the existence of such a
coordinate change, we only need to know that the degree-zero
centralizer of $t$ had this form for some different choice of
$z_i$. To prove the latter we can alternatively follow \cite[Lemma
3.2]{Kalss} as in the proof of Theorem \ref{t:main-sl}.(ii) in \S
\ref{ss:t-main-sl-pf} below; the idea being that $\nabla_u :=
\{t,u\}^{-1} \{t,-\}$ defines a flat connection on $\cO(\bC^\times
\times \Delta^{2n-1})^{\bC^\times}$ as an $\bC[\![u]\!]$-module,
together with its filtration by powers of the ideal of $(\bC^\times
\times \{0\})/\bC^\times$. Hence $\cO(\bC^\times \times \Delta^{2n-1})
\cong \bC[\![u]\!] \hat \otimes (\cO(\bC^\times \times
\Delta^{2n-1})^{\bC^\times})^\nabla$, and the second factor consists
of the degree-zero elements commuting with $t$.  It then follows that
the latter factor is a power-series algebra generated by elements
which map to the same basis for its cotangent space as the original
elements $z_1,\ldots,z_{2n-2}$. We omit further details, since we give
an explicit construction of the coordinate change needed in the next
paragraph.

To explicitly find the needed modification of the $z_i'$, we proceed
inductively.  We know that $\{t,z_i'\} \in J$ for all $i$.  Suppose
that $\{t,z_i'\} \in J^m$ for some $m \geq 1$. Then we can make the
coordinate change $z_i' \mapsto z_i'' := z_i' - t^{k-1}\int \{t,z_i'\}
du$, where $\int f du$ denotes the unique antiderivative of $f$ by $u$
which is a multiple of $u$.  Since $\{t,z_j'\} \in J$ for all $j$,
$\{t,u\}-t^{1-k} \in J$, and $\{t,z_i'\} \in J^m$, it follows from the
definition that $\{t, z_i''\} \in J^{m+1}$, as desired.  Moreover,
$z_i''-z_i' \in J^{m+1}$, so we can iterate the aforementioned
coordinate change $z_i' \mapsto z_i''$ and the result will converge to
a function $\overline{z_i'}$ such that $\{t,\overline{z_i'}\} =
0$. Doing this for all $i$, we obtain a change of coordinates $z_i'
\mapsto \overline{z_i'}$ which is the identity modulo $J$, so that
$\{t,\overline{z_i'}\} = 0$ and the original assumptions remain
satisfied. From now on we assume that this is done, so that
$\{t,z_i'\}=0$ for all $i$.

Now, the entire subalgebra $\bC[t,t^{-1}][\![z_1,\ldots,z_{2n-2}]\!]$
of elements commuting with $t$ must be closed under the Poisson
bracket. Since the bracket has degree $-k$, it follows that the
degree-zero subalgebra $\bC[\![z_1,\ldots,z_{2n-2}]\!]$ is closed
under the operation $\{-,-\}' := t^k\{-,-\}$. Moreover, this operation
defines on it a Poisson bracket which is symplectic.  By the ordinary
formal Darboux theorem, we can make a change of variables of
$\bC[\![z_1,\ldots,z_{2n-2}]\!]$ so that $\{-,-\}'$ is in standard
form.  As a result, $\{z_i',z_j'\}$ are all what we desire.

  Therefore, all of the brackets except for those with $u$ are as
  desired. To finish, we will show that there is a change of variables
  $u \mapsto u'$ so that $\{u',z_i'\}=0$ for all $i$, so that $u'
  \equiv u \pmod{J}$ and $u'$ has weight zero. We could also apply
  \cite[Lemma 3.2]{Kalss} for this, to show that $\cO(\bC^\times
  \times \Delta^{2n-1})^{\bC^\times}$ decomposes as a tensor product of
  $\bC[\![z_1,\ldots,z_{2n-2}]\!]$ and its centralizer, and pick $u'$
  as a generator of the latter.  This is what we must do in \S
  \ref{ss:t-main-sl-pf} below.  However, we give in the next paragraph a
  shortcut argument using that every Poisson vector field on a
  symplectic disk is Hamiltonian.

  Let $\xi_u$ be the Hamiltonian vector field $\{u,-\}$. Since
  $\{u,t\}=t^{1-k}$, the Jacobi identity implies that $t^k\xi_u$
  preserves $\bC[\![z_1,\ldots,z_{2n-2}]\!]$. Moreover, the Jacobi
  identity again proves that it is a Poisson vector field on $\Spf
  \bC[\![z_1,\ldots,z_{2n-2}]\!]$ where the latter is equipped with
  the Poisson structure $\{-,-\}'$ discussed above.  Since every Poisson
  vector field on a symplectic disk is Hamiltonian, $t^k\xi_u =
  t^k\{f,-\}$ for some $f \in \bC[\![z_1,\ldots,z_{2n-2}]\!]$ which we
  can assume to be in the augmentation ideal.  Then we can make the
  substitution $u \mapsto u-f$ which will have the desired property.

\subsection{Quantization}\label{ss:quant1}
The above theorem has the following consequence.  Suppose that $X=
\Spec(\cO(X))$ is an affine Poisson variety (such as an affine
symplectic variety), or that $X
= \Spf \cO(X)$ is the formal neighborhood of some subvariety of an
affine Poisson variety.

When $X = \Spec(\cO(X))$,
recall that a quantization $A_\hbar$ of $X$ (or equivalently of
$\cO(X)$) is a $\bC[\![\hbar]\!]$-algebra equipped with an algebra
epimorphism $\pr: A_\hbar \to \cO(X)$ with kernel $\hbar A_\hbar$, such
that $A_\hbar$ is isomorphic, as
a $\bC[\![\hbar]\!]$-module,  to $\cO(X)[\![\hbar]\!]  =
\{\sum_{m \geq 0} a_m \hbar^m \mid a_m \in \cO(X)\}$, and satisfies
the following additional axiom.  For any $f \in \cO(X)$, let $\tilde f
\in A_\hbar$ be an arbitrary lift.  Then $A_\hbar$ is a quantization
if it additionally satisfies:
\begin{gather}
  % \tilde a \star \tilde b \equiv \widetilde{ab} \pmod \hbar, \forall a,b \in \cO(X), \\
  \tilde a \star \tilde b - \tilde b \star \tilde a \equiv \hbar \widetilde
  {\{a,b\}} \pmod \hbar^2, \forall a,b \in \cO(X);
\end{gather}
it is easy to see that the axiom holds for one choice of lift if and only
if it holds for all choices.  We can make the same definition when 
$X = \Spf \cO(X)$ is the formal spectrum of a complete topological Poisson
algebra, where now $A_\hbar$ is also a complete topological algebra.

Next, suppose that $X$ has a $\bC^\times$-action such that the Poisson
structure has degree $-k$. Then we can ask that the quantization
$A_\hbar$ be compatible with the action.  This means that $A_\hbar$
admits an infinitesimal action of $\bC^\times$, i.e., an Euler
derivation $\Eu: A_\hbar \to A_\hbar$, which satisfies
$\Eu(\hbar)=k\hbar$, and such that the induced action on
$A_\hbar/\hbar A_\hbar \cong \cO(X)$ agrees with the original Euler
derivation ($\Eu(x)=|x| \cdot x$ when $x$ is homogeneous).  This implies
that $A_\hbar$ is the $\hbar$-adic completion of a graded algebra
having $|\hbar|=k$. This graded algebra can be recovered as the
$\bC^\times$-finite part, $A_\hbar^f \subseteq A_\hbar$, defined as
the collection of elements $a \in A_\hbar$ such that $\{\Eu^k(a)\}$
spans a finite vector space.  Slightly abusively, we will refer to
such $A_\hbar$ themselves as graded algebras and call an isomorphism
$A_\hbar \to B_\hbar$ graded if it is compatible with the
infinitesimal $\bC^\times$-action, i.e., the Euler derivations.

Let $\Weyl_{\hbar,k}(\bC^{2n-2})$ be the graded $\bC[\hbar]$-algebra which is
generated by elements $z_1', \ldots, z_{2n-2}'$, with relations
$[z_{2i-1}',z_{2i}']=\hbar$ for $1 \leq i \leq n-1$ and all other
$[z_i',z_j']$ zero, and equipped with the grading with $|\hbar|=k$,
$|z_{2i-1}'|=k$ and $|z_{2i}|=0$ for all $1 \leq i \leq n-1$.  (Note
that, if we set $z_{2i-1} := t^{-k}z_{2i-1}'$ and $z_{2i} := z_{2i}'$
for $1 \leq i \leq n-1$, we recover generators of degree zero, and
this is consistent with the notation of \S \ref{ss:without-moser}.)
\begin{theorem}\label{t:q-unique}
  The unique $\bC^\times$-compatible quantizations of $\bC^\times
  \times \bA^{2n-1}$ and of $\bC^\times \times \Delta^{2n-1}$,
  equipped with the standard symplectic structures \eqref{e:ss-deg-k},
  are given by the completions of the graded algebra (for $|t|=1$ and $|u|=0$):
\begin{equation}\label{e:gr-quant}
\caD_{\hbar,k}(\bC^\times \times \bA^{n-1}) := \bC\langle t,u \rangle[\hbar][t^{-1}] /
 ([t,u]-\hbar t^{1-k}) \otimes_{\bC[\hbar]}
\Weyl_{\hbar,k}(\bC^{2n-2})
\end{equation}
with respect to the ideal $(\hbar)$ and the ideal
generated by $\hbar, u, z_1', \ldots,
z_{2n-2}'$, respectively.
\end{theorem}
Note that, for $k=1$, $\caD_{\hbar,k}(\bC^\times \times \bA^{n-1})$ is
the Rees algebra of the ring of differential operators on $\bC^\times
\times \bA^{n-1}$; for general $k$, the ring of differential operators
is recovered by setting $\hbar=1$.  Also, the completion of
\eqref{e:gr-quant} is indeed a completed tensor product of
completions.

Putting the theorem together with Corollary \ref{c:ss-nf} immediately
yields the following.  We say that two quantizations of $\cO(X)$ are
\emph{equivalent} if they are isomorphic as
$\bC[\![\hbar]\!]$-algebras via a continuous isomorphism that is the
identity modulo $\hbar$ (this is also known as gauge equivalence). In
the $\bC^\times$-compatible case, we can ask for a
graded equivalence, i.e., an equivalence preserving the $\bC^\times$-action.
\begin{corollary}
  Let $X$ and $x$ be as in Corollary \ref{c:ss-nf}. Then the unique
  $\bC^\times$-compatible quantization of $\hat X_{\bC^\times \cdot
    x}$ up to graded equivalence
 is given in the theorem. In
  particular, every compatible quantization of $X$ restricts to this
  one, up to graded equivalence.
\end{corollary}
\begin{proof}[Proof of Theorem \ref{t:q-unique}]
  % By Theorem \ref{t:e-db} we can replace $\hat X_{\bC^\times \cdot x}$
  % by $\bC^\times \times \Delta^{2n-1}$ with the standard
  % symplectic structure \eqref{e:ss-deg-k} and hence Poisson bivector
  % \eqref{e:ss-deg-k-pb}.  It is immediate to verify that the
  % completion of \eqref{e:gr-quant} quantizes this.  To show uniqueness
  This follows from the same argument as the proof in \S
  \ref{ss:without-moser} of Theorem \ref{t:e-db}, just like the case
  of uniqueness of the quantization of the formal disk $\Delta^{2n}$
  itself.  Namely, let $Y$ be either $\bC^\times \times \bA^{2n-1}$ or
  $\bC^\times \times \Delta^{2n-1}$. For any quantization $A_\hbar$ of
  $Y$, fix a $\bC[\![\hbar]\!]$-module isomorphism $A_\hbar \cong
  \cO(Y)[\![\hbar]\!]$, compatible with the projection to $\cO(Y)$.
  We then can view the quantization as given by an associative,
  continuous multiplication $\star$ on $\cO(Y)[\![\hbar]\!]$.
  
  Inductively, assume that, for $f,g \in \cO(Y)$,
  %$\{t,u,z_1',\ldots,z_{2n-2}'\}$, then $v \star w - w \star v \equiv
  %\hbar \{u,v\} \pmod{\hbar^n}$ 
  then $f \star g - g \star f \equiv \hbar \{f,h\} \pmod{\hbar^m}$ for
  some $m \geq 2$. (It is clear, for the base case, that we can do
  this when $m=2$.) Then, as in the proof of Theorem \ref{t:e-db},
  first we modify $z_i$ so that they commute with $t$ modulo
  $\hbar^m$.  The reader can safely skip the remainder of the proof,
  since it closely follows \S \ref{ss:without-moser}; we provide the
  details for completeness.

  Let $[f,g]_\star := f \star g - g \star f$. We iteratively apply the
  gauge transformation $z_i \mapsto z_i - \hbar^{-1} \int [t,z_i]_\star
  du$ (fixing the other coordinates $t,u$, and $z_j$ for $j \neq i$).
  If $\hbar^{-m}[t,z_i]_\star \pmod{\hbar}$ vanishes to order $k$ at
  the augmentation ideal $J=(u,z_1,\ldots,z_{2n-2})$, then the
  transformation is the identity modulo $\hbar^{m-1} J^{k+2}$ and, after
  applying it, $\hbar^{-m}[t,z_i]_\star\pmod{\hbar}$ vanishes to order
  $k+1$. So this sequence of transformations converges and,
  afterwards, $[t,z_i]_\star \equiv 0 \pmod{\hbar^{m+1}}$, without
  affecting $[t,z_j]_\star$ for $j \neq i$ or any commutators modulo
  $\hbar^2$.  Applying this to all the $z_j$ means that we can assume
  $[t,z_j]_\star \equiv 0 \pmod{\hbar^{m+1}}$ for all $j$.

  Let $B$ be the subalgebra of $\cO(Y)$ which is generated by the
  $z_i$, i.e., $\bC[z_1,\ldots,z_{2n-2}]$ or
  $\bC[\![z_1,\ldots,z_{2n-2}]\!]$.  As a consequence of the above,
  the degree-zero part of the centralizer of $t$ modulo $\hbar^{m+1}$
  in $A_\hbar/\hbar^{m} A_\hbar$ under the star product is
  $B[\hbar]/(\hbar^m)$.  Then, for each $i$, the vector field $\xi :=
  t^k\hbar^{-m}[z_i,-]_\star \pmod{\hbar}$ is Poisson on $\bA^{2n-2}$
  or $\Delta^{2n-2}$ with respect to the Poisson backet $\{-,-\}'=
  t^k\{-,-\}$) and hence Hamiltonian. Thus $\xi=t^k\{f,-\}$ for some
  $f \in B$ and we can perform a gauge transformation $z_i \mapsto z_i
  - \hbar^{m-1}f$ (fixing the other coordinates) so that
  $[z_i,z_j]_\star \equiv 0 \pmod{\hbar^{m+1}}$ for all $j$.  (Note
  that this is simpler than what we had to do in the Darboux theorem
  itself, owing to $m \geq 2$).

  After all of this, the star-product commutators modulo $\hbar^{m+1}$
  are as desired amongst the coordinates $t,z_1,\ldots,z_{2n-2}$,
  and we are left only with commutators involving the coordinate $u$. Now, we have
  \[ [t,[u,z_i]_\star]_\star \equiv [[t,u]_\star,z_i]_\star \equiv 0
  \pmod{\hbar^{m+2}},
\]
and hence $t^k\hbar^{-m}[u,z_i]_\star \pmod{\hbar}$ lies in $B$.
Therefore, as before, $\xi:=t^k\hbar^{-m}[u,-]_\star \pmod{\hbar}$ is
a Poisson and hence Hamiltonian vector field on $B$, so
$\xi|_B=t^k\{f,-\}$ for some $f \in B$. Then, after performing the
gauge transformation $u \mapsto u-\hbar^{m-1} f$, we will obtain
$[u,z_i]_\star \in \hbar^{m+1} A_\hbar$, as desired.

After all of these steps, we only need to make a further gauge
transformation so that $[t,u] \equiv \hbar t^{k-1}
\pmod{\hbar^{m+1}}$.  Since $t$ and $u$ commute with the $z_i$ modulo
$\hbar^{m+1}$, the centralizer of the $z_i$ modulo $\hbar^{m+1}$ is
$\bC[t,t^{-1}][u]$ or $\bC[t,t^{-1}][\![u]\!]$ (depending on whether
we used $\bA^{2n-1}$ or $\Delta^{2n-1}$).  Since $[t,u]_\star$ has
degree $1-k$, we obtain that $\hbar^{-m}(t^{k-1}[t,u]_\star-\hbar)
\pmod{\hbar}$ is in $\bC[u]$ or $\bC[\![u]\!]$. We can therefore make
the substitution $u \mapsto u-\hbar^m \int
\hbar^{-m}t^{k-1}([t,u]_\star-\hbar) du$, which will ensure that
$[t,u] \equiv \hbar t^{1-k} \pmod{\hbar^{m+1}}$, and all brackets will
be as desired modulo $\hbar^{m+1}$, completing the induction.
% following the
%   usual argument of the formal Darboux theorem, we can modify each of
%   $t,u,z_1',\ldots,z_{2n-2}'$ by adding elements of $\hbar^n \cO(Y)$
%   so that this holds modulo $\hbar^{n+1}$.  Or, we can follow the
%   procedure of Remark \ref{r:darboux} (explained in the case of $u$,
%   but the other cases are similar).
\end{proof}
\begin{remark}
  When $k > 0$, one can instead ask about filtered quantizations.
  Namely, for $B$ a graded Poisson algebra with bracket of degree
  $-k$, one can ask for a filtered algebra $A = \bigcup_m A_{\leq m}$
  whose associated graded algebra $\gr(A) = \bigoplus_m \gr_m(A) :=
  A_{\leq m} / A_{\leq m-1}$ is isomorphic to $B$ as an algebra, such
  that $[A_{\leq m}, A_{\leq n}] \subseteq A_{\leq m+n-k}$ and
  $\gr_{m+n-k}[a,b] = \{\gr_m a, \gr_n b\}$ for $a \in A_{\leq m}$ and
  $b \in A_{\leq n}$.  If one asks also that $A$ be complete with
  respect to the descending part of the filtration, i.e., $A = \lim_{m
    \to -\infty} A/A_{\leq m}$ as a vector space (which is a trivial
  condition if $B$ is nonnegatively graded), then one can generalize
  the results above to show that there is a unique filtered
  quantization as well.

  Note that, for $k=1$, this is a formal consequence, since then
  filtered quantizations are equivalent to $\bC^\times$-compatible
  quantizations: to a filtered algebra we associate its completed Rees
  algebra; the reverse direction is given by taking
  $\bC^\times$-finite vectors and then setting $\hbar=1$.  But for $k
  \geq 2$, not all filtered quantizations can be obtained from
  $\bC^\times$-compatible quantizations: for example, for $k=2$ and
  $B=\bC[t,y]$ with $|y|=2, |t|=1$, and $\{t,y\}=t$, one can take the
  filtered quantization $\bC\langle t,y \rangle / ([t,y]-t-1)$, which
  cannot be obtained from a compatible quantization (but it can, in
  view of the uniqueness result above, after inverting $t$ and
  completing with respect to the filtration; then, in the above
  coordinates, $y=(t^2+t)u$).  (One can, however, make a more direct
  link by working over $\bC[\![\hbar^{1/k}]\!]$: then, every filtered
  quantization can be obtained from a $\bC[\![\hbar^{1/k}]\!]$-algebra
  with a compatible $\bC^\times$ action.)
  % $\bC\langle t,u \rangle / ([t,u]=t^{1-k} + t^{-k})$ can not be
  % obtained in this way (here $n=2$ and $\Spec B = \bC^\times \times
  % \bA^1$ with the standard symplectic structure of degree $-k$).

  Similarly, in all later results in this paper on quantization, one
  can obtain analogous results for filtered algebras by passing to the
  $\bC^\times$-finite vectors and setting $\hbar=1$, in the case when
  $k$ is positive. We will not mention this further.
\end{remark}

\section{Equivariant Darboux-Weinstein theorems}\label{s:main}
In this section, we give the main theorems (\ref{t:main-sl},
\ref{t:conv}, and \ref{t:main-sr}), their application to
$\caD$-modules on Poisson varieties (Corollary \ref{c:main-dmod} and
Remark \ref{r:main-dmod}), and their quantization (Theorem
\ref{t:quant}).  We begin in the first subsection with the statements
of the main theorems and corollaries, and prove only the corollaries.
We then give an application to $\caD$-modules, followed by the proof of
Theorem \ref{t:pre-main} and the main theorems, and finally discuss
quantization.  The proofs of the main theorems can be omitted on a
first reading.

\subsection{Main results}

Recall that, in the non-equivariant setting, the formal
Darboux-Weinstein theorem (\cite{We}; see also \cite[Proposition
3.3]{Kalss}) says that, for $Y \subseteq X$ an a symplectic leaf of a
Poisson variety $X$ and $y \in Y$ a point, then the formal
neighborhood $\hat X_y$ is Poisson isomorphic to a product, with $S$ a
formal Poisson scheme (with a single closed point and Poisson bivector
vanishing there):
\begin{equation} \label{e:fwd} \hat X_y \cong \hat Y_y \times S \cong
  \Delta^{\dim Y} \times S.
\end{equation}
In the $\bC^\times$-equivariant setting, this no longer holds when $y$
is not a fixed point, as in the following examples 
(although one can easily see it does
hold for elliptic fixed points as in Remark \ref{r:efp}).
\begin{example}\label{ex:counterex1}
  Consider $X = \bC^{\times} \times \Delta^2 = \Spf
  \bC[t,t^{-1}][\![u,z]\!]$ with Poisson bivector $\partial_u \wedge
  (t\partial_t+\partial_z)$ of degree zero (note that the resulting
  bracket automatically satisfies the Jacobi identity since the
  Poisson bivector has rank two). Then the Poisson center consists of
  those functions in $\bC[t,t^{-1}][\![z]\!]$ annihilated by the
  Hamiltonian vector field $\xi_u = t \partial_t + \partial_z$, i.e.,
  $\bC[t e^{-z},t^{-1} e^z]$, whose spectrum is $\bC^\times$
  with the dilating $\bC^\times$ action.  On the other hand, the
  Poisson center of the product of the symplectic $\bC^\times$-variety
  $\bC^{\times} \times \Delta$ and $\Delta$ is the functions on the
  latter factor, $\bC[\![z]\!]$, which is not isomorphic to $\bC[t
  e^{-z}, t^{-1} e^z]$, and is equipped with the trivial $\bC^\times$
  action.

  We remark that, to get a trivial Poisson center rather than $\bC[t
  e^{-z}, t^{-1} e^z]$ as above, we could have instead used the Poisson
  bivector $\partial_u \wedge (t \partial_t + \gamma z \partial_z)$
  for $\gamma \in \bC$ irrational: then the Poisson center is trivial
  since the differential equation $\gamma z \partial_z (f) = m f$ has
  no solutions with $m \in \bZ$ and $f \in \bC[\![z]\!]$ (i.e.,
  $z^{m/\gamma} \notin \bC[\![z]\!]$).
\end{example}
\begin{example}\label{ex:counterex2}
  We can also give a singular example where the Poisson structure is
  nondegenerate on the smooth locus (in particular, this implies it is
  generically symplectic). Let $X = \bC^\times \times \Delta 
  \times Z$ where $Z$ is the formal neighborhood of the origin of the
  hypersurface $x^3+y^3+z^3=0$ in $\bC^3$ (i.e., the cone over a
  smooth genus one curve in $\bC \mathbf{P}^2$).  Let $\xi$ be the
  usual Euler vector field on $Z$, i.e., $x \partial_x + y \partial_y
  + z \partial_z$.  Equip $Z$ with the Jacobian Poisson structure,
\[
\pi_Z = 3 \bigl( x^2 \partial_y \wedge \partial_z + y^2 \partial_z
\wedge \partial_x + z^2 \partial_x \wedge \partial_y \bigr).
\]
Now, $\xi$ is a Poisson vector field, i.e., $L_{\xi}(\pi_Z)=0$, where
$L_{\xi}$ is the Lie derivative, or equivalently, $[\xi,\pi_z]=0$,
where $[-,-]$ is the Schouten-Nijenhuis bracket.  Therefore, the
following bivector $\pi$ is Poisson (i.e., $[\pi,\pi]=0$):
\begin{equation}
  \pi = \partial_u \wedge (t \partial_t +\xi) + \pi_Z.
\end{equation}
On the other hand, we claim that this is not equivalent via change of
coordinates to a product of formal Poisson schemes $(\bC^\times 
\times \Delta) \times Z$.  Indeed, if it were, 
% taking the degree
% zero part of the centralizer of $t$, we get that $Z = \Spf
% \bC[\![x,y,z]\!] / (x^3+y^3+z^3)$ (as a subalgebra of all functions on
% $(\bC^\times \times \Delta) \times Z$), and then 
there would
have to be a degree zero coordinate $u'$ such that $\{u',t\}=t$ and
$u'$ commutes with $x,y$, and $z$.
%(note that the degree zero condition
%is unnecessary, since taking any homogeneous component of such a $u'$
%of nonzero degree we would get a Poisson central element, and the
%Poisson center is zero since the Poisson bivector is generically
%nondegenerate).  
The first condition shows that $u' \in u + f$ for
some $f \in \bC[\![x,y,z]\!]$, and the second condition shows that
$\{f,g\}=-\xi(g)$ for all $g \in \bC[\![x,y,z]\!]$.  This means that
$\xi$ is a Hamiltonian vector field on $\bC[\![x,y,z]\!]$. But that is
false, since every Hamiltonian vector field has positive degree,
whereas $\xi$ has degree zero.  This is a contradiction.
\end{example}
The preceding example demonstrates exactly what goes wrong, however,
as we prove in the following theorem.  Let $X$ be a
$\bC^\times$-Poisson variety with a homogeneous Poisson bivector, $Y
\subseteq X$ be a $\bC^\times$-invariant symplectic leaf, and $y \in
Y$ be a point with minimal stabilizer under the $\bC^\times$-action.
If the stabilizer of $y$ is in fact trivial, then $\hat Y_{\bC^\times
  \cdot y}$ has a standard symplectic structure \eqref{e:ss-deg-k}. In
general, we obtain a symplectic structure \eqref{e:ss-deg-k-nf}.
% Applying the formal Darboux-Weinstein theorem to $\hat X_y$, write
% $\hat X_y \cong \hat Y_y \times S$, as formal Poisson schemes,
% for
% some $S$ with some Poisson structure $\pi_S$.
Let $\hat X_{\bC^\times \cdot y}$ denote the formal neighborhood of
$\bC^\times \cdot y$ in $X$ (equivalently, we can replace $X$ with any
open affine subvariety in which the punctured line $\bC^\times \cdot
y$ is closed.) Let $t \in \cO_{\hat X_{\bC^\times \cdot y}}$ be the
coordinate along the line $\bC^\times$, i.e., $t(c \cdot y) = c$ for
all $c \in \bC^\times$.

For every (formal) scheme $Z$, let $\Omega_Z$ be the sheaf of K\"ahler
differentials. If $Z^\circ$ is the smooth locus with inclusion $j:
Z^\circ \to Z$, then 
%%following \cite{Fer-chdfcas, Fer-cfdssac}
%% we define
we define $\tilde \Omega_Z := j_* \Omega_{Z^\circ}$, the underived pushforward of $\Omega_{Z^\circ}$ (this differs from the definition of
of \cite{Fer-chdfcas,Fer-cfdssac} where one takes $\Omega_{Z}$ modulo torsion\footnote{Thanks to a referee for clarifying this point.}). 
We will also use the complexes of de Rham differentials, $\Omega_Z^\bullet = \wedge_{\cO_X}^\bullet \Omega_Z$ and $\tilde \Omega_Z^\bullet := j_* \Omega_{Z^\circ}^\bullet$.
%In
%the case that $Z$ is normal that we will consider, 
%%$\tilde \Omega_Z$ to be
%can be identified with 
%%the quotient of $\Omega_Z$ by torsion.  (Note that this does not
%%coincide with the definition $j_* \Omega_{Z^\circ}$ often found in the
%%literature: if $Z$ is normal, then $j_* \Omega_{Z^\circ}$ is the
%%reflexive envelope of what we define as $\tilde \Omega_Z$).  
Moreover,
for us $Z$ will be affine, so $\Omega_Z$ and $\tilde \Omega_Z$ can be
identified with their global sections.  Assume that $Z$ is normal
affine.  Then, $H^1(\tilde \Omega_Z^\bullet) = H^1 \Gamma(Z^\circ,
\Omega^\bullet_{Z^\circ})$, which equals global closed algebraic one-forms on
$Z^\circ$ mod global exact algebraic one-forms.  When $Z$ is a complex
algebraic variety, this embeds into the topological de Rham cohomology
$H^1(Z^\circ)$, because if an algebraic one-form $\alpha$ on a smooth
variety is the differential of a $C^\infty$ function $f$, then $f$
must itself be algebraic. In particular, if $H^1(Z^\circ) = 0$, then
$H^1(\tilde \Omega^\bullet_Z) = 0$.

For $Z$ Poisson, let $H(Z)$ denote the Lie algebra of Hamiltonian
vector fields on $Z$, i.e., $H(Z) = \{\xi_f \mid f \in Z\}$, with
$\xi_f(g) := \{f,g\}$. Let $P(Z)$ denote the Lie algebra of Poisson
vector fields, i.e., vector fields $\xi$ which are Lie derivations of
the Poisson bracket (equivalently, $[\xi, \pi] = 0$, using the
Schouten-Nijenhuis bracket, with $\pi$ the Poisson bivector). We
always have $H(Z) \subseteq P(Z)$.  When $Z$ is normal, affine, and
symplectic on the smooth locus, then $P(Z)/H(Z) \cong H^1(\tilde
\Omega^\bullet_Z)$, since for $\omega$ the generic symplectic form,
then $i_\xi \omega$ is a regular one-form on $Z^\circ$ for every
vector field $\xi$, and $\xi$ is Hamiltonian or Poisson if and only if
$i_\xi \omega$ is closed or exact, respectively.

  % Below we say that the formal scheme $\hat X_y$ is normal if it is
  % normal as an ordinary scheme, i.e., it is integrally closed (by
  % Zariski's main theorem, this is equivalent to a Zariski
  % neighborhood of $y$ being normal).

  % By Grothendieck's theorem on comparison of algebraic and
  % topological de Rham cohomology, $H^*(Z^\circ) =
  % \mathbf{H}(Z^\circ, \Omega_{Z^\circ})$, the hypercohomology.  In
  % particular, this gives a decomposition
% \begin{equation}
% H^1(Z^\circ) \cong H^1(\tilde \Omega_Z) \oplus R^1 \Gamma(Z^\circ, \cO_{Z^\circ}),
% \end{equation}
% as the degree one part of the spectral sequence computing
% hypercohomology.

\begin{theorem}\label{t:main-sl}
  Assume that $y$ has trivial stabilizer under $\bC^\times$.
  In a formal neighborhood of $\bC^\times \cdot y$, with all
  isomorphisms $\bC^\times$-equivariant,
\begin{enumerate}
\item[(i)] $\hat Y_{\bC^\times \cdot y} \cong \bC^{\times} \times
  \Delta^{\dim Y - 1}$ with a standard homogeneous symplectic
  structure \eqref{e:ss-deg-k} of degree $k$.
\item[(ii)] 
 Forgetting the Poisson structure,
  % as formal $\bC^\times$-schemes,
\begin{equation}
\hat X_{\bC^\times
    \cdot y} \cong \hat Y_{\bC^\times \cdot y} \times S,
\end{equation} 
for some formal scheme $S$ with the trivial $\bC^\times$
action. The projection $\hat X_{\bC^\times \cdot y} \to \hat Y_{\bC^\times \cdot y}$ is Poisson. The second projection is not Poisson, but 
$t^k \{-,-\}$ restricts on $\cO(S)$
%  $\cO(S)$ identifies with the algebra of degree-zero elements
% commuting with $t,z_1,\ldots,z_{2n-2}$.
% Moreover,
to a Poisson bracket.
\item[(iii)] We can choose coordinates as above so that
  the Poisson bivector on $\hat X_{\bC^\times \cdot y}$ has the form
\begin{equation}\label{e:hx-pbv}
  \pi_{\dim Y,-k} + t^{-k}\partial_u \wedge \xi + t^{-k}\pi_S,
\end{equation}
for $\pi_S$ the Poisson bivector on $S$, and $\xi$ a vector field
on $S$ satisfying $[\xi,\pi_S] = k\pi_S$.
\end{enumerate}
\end{theorem}
More generally, in the case where the stabilizer of $y$ is arbitary, we deduce:
\begin{corollary}\label{c:main-sl-nf}
\begin{enumerate}
\item[(i)] $\hat Y_{\bC^\times \cdot y} \cong \bC^{\times} \times
  \Delta^{\dim Y - 1}$ with a standard homogeneous symplectic
  structure \eqref{e:ss-deg-k-nf} of degree $k$.
\item[(ii)]  Let $\tilde {\hat
    Y}_{\bC^\times \cdot y} \to \hat Y_{\bC^\times \cdot y}$ be the
  $\ell$-fold \'etale cover obtained by adjoining $t^{1/\ell}$.  Then,
  there is a canonical formal Poisson scheme $S$ with a $\bZ/\ell$
  action and a continuous $\bC^\times$-equivariant isomorphism
\begin{equation}
\hat X_{\bC^\times \cdot y} \cong (\tilde  {\hat Y}_{\bC^\times \cdot y} \times S)/(\bZ/\ell),
\end{equation}
which is Poisson when the RHS is equipped with
a structure as in \eqref{e:hx-pbv}. 
% Alternatively, in coordinates \eqref{e:ss-deg-k-nf} for $\hat
% Y_{\bC^\times \cdot y}$, let $\cO(T) \subseteq \hat \cO(X)_{\bC^\times
%   \cdot y}$ be the centralizer of $t,z_1,\ldots,z_{2n-2}$; then we get
% a $\bC^\times$-isomorphism,
% \begin{equation}\label{e:prod-t}
%   \hat X_{\bC^\times \cdot y} \cong \hat Y_{\bC^\times \cdot y} \times_{\bC^\times} T,
% \end{equation}
% where the fibered product is with respect to the subalgebra
% $\bC[t,t^{-1}]$ of $\cO(\hat Y_{\bC^\times \cdot y})$ and $\cO(T)$.
% The Poisson bivector on $\hat X_{\bC^\times \cdot y}$ is then given by
% \begin{equation}\label{e:hx-pbv-t}
%   \pi_{\hat Y_{\bC^\times \cdot y}} + \pi_T + t^{-k} \partial_u \wedge \xi,
% \end{equation}
% with $\xi$ a degree-zero Poisson vector field on $T$.
\end{enumerate}
\end{corollary}
\begin{proof}
  Part (i) is an immediate consequence of Corollary \ref{c:ss-nf}.

Part (ii) follows directly from Theorem \ref{t:main-sl}.  Namely, when
we adjoin $t^{1/\ell}$, then the latter coordinate again has degree
$1$, so we can apply the theorem and get a product decomposition for
the \'etale $\ell$-fold cover.
\end{proof}
% \begin{remark}
%  The final version of the statement is convenient, even when $\ell=1$,
% for quantization: see Theorem \ref{t:quant}.
% \end{remark}
For simplicity, we restrict for the remainder of the subsection
 to the case $\ell=1$.
% , although the statements
% all admit straightforward generalizations, via part (4) above, to the
% case of arbitrary $\ell$.
%% Let us continue to define $T$ and $S$ as in the theorem, so $S =
%% T/\!/\bC^\times$, and in the case that $y$ has trivial stabilizer
%% under $\bC^\times$, then $T = S \times \bC^\times$.
To avoid confusion, when $S$
%% or $T$
is a formal Poisson scheme, we let
the Hamiltonian vector fields of functions on $S$ 
%%and $T$ 
be denoted
as $\xi^S_f$.
%% and $\xi^T_f$, respectively.
% When we equip $S$ with the Poisson bivector $\pi_S = t^{-k}
% \pi_X|_{\cO_S}$, we let the corresponding Hamiltonian vector fields
% on $S$ be denoted by $\xi^S_f = t^{-k} \xi_f|_{\cO_S}$, to avoid
% confusion with Hamiltonian vector fields on $X$.  We can also give a
% converse result (recall that $\xi_f$ is the Hamiltonian vector field
% $\{f,-\}$ corresponding to a function $f$):$
\begin{theorem}\label{t:conv}
  Let $m \geq 1$ and fix $k \in \bZ$.  Given two formal Poisson
  schemes $S, S'$ with Poisson bivectors $\pi_S$ and $\pi_{S'}$,
  equipped with vector fields $\xi, \xi'$ satisfying
  $[\xi,\pi_S]=k\pi_S$ and $[\xi',\pi_{S'}]=k\pi_{S'}$, the formal
  Poisson schemes $\bC^\times \times \Delta^{2m-1} \times S$
  and $\bC^\times \times \Delta^{2m-1} \times S'$ with
  bivectors \eqref{e:hx-pbv} are $\bC^\times$-equivariantly Poisson
  isomorphic (with $\bC^\times$ acting only on the first factor) if
  and only if there is formal Poisson isomorphism $S \to S'$ taking
  $\xi$ to $\xi'+\xi^S_f$ for some $f \in \cO_{S'}$.
%   More generally, suppose $T, T'$ are $\bC^\times$-Poisson schemes
%   with central homogeneous coordinate functions $t, t'$ both of degree
%   $k \geq 1$, equipped with Poisson vector fields $\xi, \xi'$.  Then,
%   formal Poisson schemes $Z,Z'$ with bivectors as in \eqref{e:hx-pbv2}
%   are
% %   $\Delta^{2m-1} \times
% %   T$ and $\Delta^{2m-1} \times T'$ with bivectors
% %   \eqref{e:hx-pbv2} are
%   $\bC^\times$-equivariantly Poisson isomorphic
%   %(with $\bC^\times$ acting only on the second factor) 
%   if and only if there is formal $\bC^\times$-Poisson isomorphism $T
%   \to T'$ taking $\xi$ to $\xi'+t^{k}\xi^T_f$ for some $f \in
%   \cO_{T'}^{\bC^\times}$.  The same is true if we replace $Z, Z'$ by
%   $\Delta^{2m-2} \times Z, \Delta^{2m-2} \times Z'$ giving
%   $\Delta^{2m-2}$ the standard symplectic structure (and the trivial
%   $\bC^\times$-action).
\end{theorem}
\begin{remark}
As we will see in the proof, the isomorphisms in the theorem
can be chosen to form a commutative square:
\[
\xymatrix{
\bC^\times \times \Delta^{2m-1} \times S \ar@{->>}[r] \ar[d]
 & S \ar[d] \\
\bC^\times \times \Delta^{2m-1} \times S' \ar@{->>}[r]
 & S'.
}
\]
\end{remark}
\begin{corollary} \label{c:ham-pbv-isom} If $S$ has the property that
  all Poisson vector fields are Hamiltonian, then all Poisson
  bivectors of the form \eqref{e:hx-pbv} (i.e., for all choices of
  $\xi$ satisfying $[\xi,\pi_S]=k\pi_S$) are $\bC^\times$-Poisson
  isomorphic.  
% The same is true for Poisson bivectors of the form
%   \eqref{e:hx-pbv2} with $S = T /\!/ \bC^\times$ and $t$ central of
%   degree $\ell \geq 1$, and $T$ an irreducible formal
%   $\bC^\times$-Poisson scheme, equipping $S$ with the Poisson bivector
%   $\pi_S = t^k \pi_T|_{S}$.
\end{corollary}
\begin{proof} If $\xi$ and $\xi'$ satisfy $[\xi, \pi_S] = k\pi_S =
  [\xi', \pi_S]$, then $\xi-\xi'$ is Poisson. Under the assumption of
  the corollary, it is also Hamiltonian. Then, by Theorem
  \ref{t:conv}, taking $S=S'$ and the identity map $S \to S'=S$, we
  conclude that the two Poisson bivectors \eqref{e:hx-pbv} given by
  $\xi$ and $\xi'$ are related by a $\bC^\times$-Poisson
  automorphism. (By the remark, it can be chosen to be compatible with
  the projection to $S$.)
  % For the last statement, we apply the above argument except with
  % $\pi_T$ instead of $\pi_S$. Then, $\xi-\xi'$ will be a degree zero
  % Poisson vector field annihilating $t$, whose restriction to $\cO_S =
  % \cO_T^{\bC^\times}$ is Hamiltonian, say with Hamiltonian function $f
  % \in \cO_S = \cO_T^{\bC^\times}$.  If $\xi_f := i_{\pi_T}(df)$ is the
  % Hamiltonian with respect to the Poisson structure on $T$, we get
  % that $\eta:=\xi-\xi' -t^k \xi_f$ annihilates both $\cO_S$ and
  % $t$. Thus, for any $g \in \cO_T$ homogeneous,
  % $\eta(g^\ell)=kg^{\ell-1}\eta(\ell)=0$.  Since $T$ is irreducible, $\eta=0$,
  % and hence $\xi - \xi' = t^k \xi_f$. The statement now follows
  % from Theorem \ref{t:conv}.
\end{proof}
\begin{corollary}\label{c:pic}
  For a fixed formal Poisson scheme $S$ and fixed $m \geq 1$ and $k
  \in \bZ$ equipped with a vector field $\eta$ satisfying $[\eta,
  \pi_S] = k\pi_S$, the set of $\bC^\times$-Poisson isomorphism
  classes of bivectors of the form \eqref{e:hx-pbv} is in bijection
  with $(\eta + P(S)/H(S))/\Aut(S,\pi_S)$.  
  % More generally, for $T$ an irreducible formal $\bC^\times$-Poisson
  % scheme with $T/\!/\bC^\times = S$ and invertible central
  % homogeneous coordinate $t$ of degree $\ell \geq 1$,
% %Poisson center $\bC[t,t^{-1}]$ with $t$ of degree $k \geq 1$,
% the set of $\bC^\times$-Poisson isomorphism classes of bivectors of
% the form \eqref{e:hx-pbv2} is in bijection with $(\eta +
% P'(S)/H(S))/\Aut(S,\pi_S)$, where $P'(S)$ is the vector space of
% Poisson vector fields on $S$ which are extendable to Poisson vector
% fields on $T$ annihilating $t$.
\end{corollary}
Note that, if $S$ is normal and symplectic on the smooth locus, we can
alternatively write $(\eta+P(S)/H(S))/\Aut(S,\pi_S)$ as
$(i_{\xi}\omega + H^1(\tilde \Omega^\bullet_S))/\Aut(S,\pi_S)$.
\begin{proof}[Proof of Corollary \ref{c:pic}]
  Let $\xi$ be any other vector field on $S$ satisfying $[\xi, \pi_S]
  = k\pi_S$.  The theorem implies that the resulting bivectors given
  by \eqref{e:hx-pbv} are isomorphic if and only if $\xi \in
  \Aut(S,\pi_S) \cdot (\eta + H(S))$.  On the other hand, $\xi-\eta
  \in P(S)$, so $\xi \in (\eta + P(S)) \subseteq \Aut(S,\pi_S) \cdot
  (\eta + P(S))$, proving the result (note that this inclusion
  $\subseteq$ is actually an equality).  
%For the second paragraph, note
%  as in the end of the proof of the previous corollary that any degree
%  zero Poisson vector field on $T$ annihilating $t$ is uniquely
%  determined by its restriction to $S$, by irreducibility of $T$.
\end{proof}
Next we explain why, when $X$ admits a symplectic resolution, all
bivectors in \eqref{e:hx-pbv} 
%and \eqref{e:hx-pbv2} 
are equivalent. This is similar to \cite[Proposition
5.5]{PS-pdrhhvnc}:
\begin{theorem}
 \label{t:main-sr} Let $X$ be a normal Poisson variety,
$Y$ a symplectic leaf, and $y \in Y$. Let $S$ be a formal
Poisson scheme appearing either in an ordinary Darboux-Weinstein
decomposition
$\hat X_y \cong \hat Y_y \times S$, or in
an equivariant one $\hat X_{\bC^\times \cdot y} \cong
\hat Y_{\bC^\times \cdot y} \times S$ of Theorem \ref{t:main-sl} for
$y$ having trivial stabilizer in $\bC^\times$.
If $X$ admits a symplectic resolution,
% $S$ is the formal neighborhood of a (normal Poisson) variety
% admitting a symplectic resolution.
then all Poisson vector fields on $S$ are Hamiltonian.
More generally, each of the following conditions implies the next:
\begin{enumerate}
\item[(i)] $X$ admits a symplectic resolution, or more generally is a
  symplectic singularity;
\item[(ii)] $X$ is generically symplectic and
there is an analytic neighborhood $U$ of $y$ such that,
for $U^\circ$ the smooth locus of $U$, we have
  $H^1(U^\circ,\bC) = 0$; 
\item[(iii)] $X$ is generically symplectic and $H^1(\tilde \Omega^\bullet_{S}) = 0$; and
\item[(iv)] All Poisson vector fields on $S$ are Hamiltonian.
\end{enumerate}
% \item Conversely, if $X$ is normal and generically symplectic and
%   $H^1(\tilde \Omega_{\hat X_{\bC^\times \cdot y}}) \neq 0$, then
%   there exists a Poisson structure on $\hat X_{\bC^\times \cdot y}$
%   of degree $d$ which is not $\bC^\times$-equivariantly isomorphic
%   to one of the form $\pi_{Y_{\bC^\times \cdot y}} + t^d \pi_S$ (for
%   all $d \in \bZ$).  More generally, if $y$ does not necessarily
%   have trivial stabilizer, then for a decomposition $\hat
%   X_{\bC^\times \cdot y} \cong \Delta^{\dim Y - 1} \times T$ as
%   in Theorem \ref{t:main-sl}, then the above holds replacing
%   conditions (iii) and (iv) with
% \begin{enumerate}
% \item[(iii')] $X$ is generically symplectic and $H^1(\tilde
%   \Omega_T) = 0$; and
% \item[(iv')] All degree-zero Poisson vector fields on $T$ annihilating
%   $t$ are Hamiltonian.
% \end{enumerate}
\end{theorem}  
In the presence of a $\bC^\times$-action on $S$, we can put the
above results together to conclude:
\begin{corollary} \label{c:ct-act} Let $X,Y,y,d$ be as in Theorem
  \ref{t:main-sl} and suppose that the stabilizer of $y$ is trivial.
  For a decomposition therein, suppose that $S$ admits a $\bC^\times$
  action such that the Poisson structure is homogeneous of degree
  $-k$, and that all Poisson vector fields on $S$ are Hamiltonian
  (e.g., if $X$ admits a symplectic resolution or is a symplectic
  singularity).  Then $\hat X_{\bC^\times \cdot y}$ is
  $\bC^\times$-equivariantly Poisson isomorphic to $(\hat
  Y_{\bC^\times \cdot y} \times S, \pi_{\hat Y_{\bC^\times \cdot
      y}} + \pi_S)$, with the new action of $\bC^\times$ on $S$.
  Moreover, in this case, $S$ is isomorphic to every slice appearing
  in an ordinary Darboux-Weinstein decomposition $\hat X_y \cong \hat
  Y_y \times S$.
\end{corollary}
\begin{proof}
  Let $\eta$ be the Euler vector field of the action of $\bC^\times$
  on $S$.  Then $[\eta, \pi_S] = -k \pi_S$. Therefore, by Theorem
  \ref{t:conv}, we can assume $\xi = -\eta$ in \eqref{e:hx-pbv}.  Let
  $B \subseteq \cO_S[t,t^{-1}]$ be the subalgebra of elements of
  weight zero with respect to the total grading by $|t|=1$ and the
  weight grading on $\cO_S$ produced by the $\bC^\times$-action.  That
  is, $B$ consists of the elements of $f \in \cO_S[t,t^{-1}]$
  annihilated by $-\eta + t \partial_t = \xi + t \partial_t$.  In
  other words, $B$ is the subalgebra of $\cO_S[t,t^{-1}]$ of elements
  commuting with $u$.  We therefore obtain $X = (\bC^\times 
  \times \Delta^{\dim Y -1}) \times \Spf B$, where now this is a
  product of formal schemes, with the Poisson structure on the first
  factor $\bC^\times \times \Delta^{\dim Y -1}$ the standard one
  of \eqref{e:ss-deg-k}. The first statement then follows from Theorem
  \ref{t:e-db}.  For the second statement, we can complete at $y$ and
  obtain an ordinary Darboux-Weinstein decomposition.  But, the $S$
  appearing in this decomposition is unique up to formal Poisson
  automorphisms, since $S$ is the centralizer of $\hat \cO_{Y,y}$, and
  two choices of subalgebra $\hat \cO_{Y,y} \subseteq \hat \cO_{X,y}$,
  i.e., of projections $\hat X_y \onto \hat Y_y$, differ by
  Hamiltonian isomorphisms (cf.~the proof of Theorem \ref{t:conv}
  below).
\end{proof}
\subsection{Consequence for the $\caD$-module on $X$}
Recall from \cite{ESdm} the $\caD$-module, $M(X)$, on $X$ which
represents solutions of Hamiltonian flow.  For simplicity, assume $X$
is affine.  Then this is defined by $M(X) = H(X) \caD_X \setminus
\caD_X$, for $H(X)$ the Lie algebra of Hamiltonian vector fields on
$X$, $\caD_X$ the canonical right $\caD$-module on $X$ such that
$\Hom(\caD_X,-)$ is the global sections functor (i.e., $\caD_X =
\iota^!  (I_X \caD_V \setminus \caD_V)$ where $\iota: X \to V$ is any
embedding into a smooth affine variety $V$ with $X$ having ideal
$I_X$), and with $H(X)$ acting on $\caD_X$ on the left via the
inclusion $H(X) \hookrightarrow \Gamma(\caD_X)=$ differential
operators on $X$.  See \cite{ESdm} for more details.

We may deduce from the preceding the following consequence on the
$\caD$-module $M(X)$.  Let $i: Y \to X$ be the inclusion of a
symplectic leaf $Y$ into $X$.
\begin{corollary}\label{c:main-dmod}
  Suppose that $X$ is a symplectic singularity and that the
  assumptions of Corollary \ref{c:ct-act} are satisfied.  Then, as
  weakly equivariant $\caD$-modules on $\hat Y_{\bC^\times \cdot y}$,
\begin{equation}\label{e:c-main-dmod}
H^0(i^* M(X))|_{\hat Y_{\bC^\times \cdot y}} \cong \Omega_{\hat
  Y_{\bC^\times \cdot y}} \otimes \operatorname{HP}_0(S),
\end{equation}
equipping $\operatorname{HP}_0(S)$ with the weight grading by the
$\bC^\times$-action on $S$.
\end{corollary}
\begin{proof}
  First, it is easy to see that $H^0 (i^* M(X))|_{\hat Y_{\bC^\times
      \cdot y}}$ is a local system on $\hat Y_{\bC^\times \cdot y}$
  (as noticed in \cite[\S 4.3]{ESdm}, this is already true restricting
  to $Y$ itself; we use here that $X$ is a union of finitely many symplectic
  leaves, which is true for every symplectic singularity).
  As noticed there, the fibers of this local system are identified
  with $\HP_0(\cO_S)$ (which is finite-dimensional).  Now, in terms of
  the decomposition of Corollary \ref{c:ct-act}, everything is weakly
  $\bC^\times$-equivariant, equipping $\HP_0(\cO_S)$ with its weight
  grading.
\end{proof}
% \begin{remark}
%   A generalization of the last two corollaries to the case where $y$
%   has stabilizer of order $k \geq 1$ is as follows (we omit the
%   proof, which is similar to those for the previous
%   generalizations): if all degree-zero Poisson vector fields on $T$
%   annihilating $t$ are Hamiltonian (e.g., if $X$ admits a symplectic
%   resolution), and $T$ admits a Poisson $\bC^\times$-action for
%   which $t$ is homogeneous of degree $k$, then we can replace $S$
%   above by $S' := \Delta \times T$, where $\Delta$ has the
%   coordinate $u$ in Theorem \ref{t:main-sl}, given weight zero. Then
%   we get a $\bC^\times$ Poisson decomposition $\hat X_{\bC^\times
%   \cdot y} \cong \Delta^{\dim Y - 2} \times S'$
   
%   %   $T \cong \bC^\times \times S'$
%   %   for $S'$ the part of degree zero with respect to the action,
%   %   and we
%   %   get a $\bC^\times$ Poisson decomposition $\hat X_{\bC^\times
%   %   \cdot
%   %   y} \cong \hat Y_{\bC^\times \cdot y} \times S'$,
%   %   equipping
%   %   $S'$ with the original grading on $T$.  Then the results go
%   %   through,
%   %   with $S$ replaced by $S'$.

%   Equivalently, taking the product of the above Poisson $\bC^\times$
%   action and the original $\bC^\times$-action, we require a
%   $\bC^\times$-action on $T$ for which $t$ has degree zero and the
%   Poisson bivector has the usual degree, $-dk$. In the case where
%   $k=1$ then this corresponds to the action on $S=T/\!/\bC^\times$
%   mentioned above, extended to act trivially on $t$.
% \end{remark}
\begin{remark}\label{r:main-dmod}
  The assumptions of Corollary \ref{c:ct-act} can be dropped, as
  follows.  All weakly $\bC^\times$-equivariant local systems on a
  formal neighborhood of a punctured line $\bC^\times \cdot y$ in a
  smooth variety are sums of weight-shifted trivial local systems.
  Therefore, if we assume nothing but the fact that $Y$ is a
  symplectic leaf closed under the $\bC^\times$-action and that $X$
  is a symplectic singularity, we know that the LHS of
  \eqref{e:c-main-dmod} must be a weakly equivariant local system on
  $\bC^\times \cdot y$ (at this point, we only need the symplectic
  singularity condition to guarantee that $M(X)$ is holonomic, which more
  generally follows if $X$ is a union of finitely many symplectic leaves). We
  can describe this local system using an ordinary Darboux-Weinstein
  slice $S$, i.e., such that $\hat X_y \cong \hat Y_y 
  \times S$.  Then, similarly to \cite[\S 5]{PS-pdrhhvnc}, the grading
  on $\HP_0(S)$ is given by an arbitrary vector field $\eta$ such that
  $[\eta, \pi_S] = -k \pi_S$ (using now the fact that all Poisson
  vector fields are Hamiltonian, from Theorem \ref{t:main-sr}). Such a
  vector field can be obtained from the Euler vector field $\Eu_X$ of
  $X$ by the projection $\pi_S: \hat X_y \to S$, namely $\eta =
  (\pi_S)_*(\Eu_X|_{\{0\} \times S})$.
\end{remark}  
% In particular, by parts \eqref{t:main-sr} and \eqref{t:main-cssl},
% if $X$ admits a symplectic resolution and $S$ has a
% $\bC^\times$-action for which $\pi_S$ has degree $d$, then $\hat
% X_{\bC^\times \cdot y} \cong \hat Y_{\bC^\times \cdot y} \times S$
% as $\bC^\times$ Poisson schemes, giving $\hat Y_{\bC^\times \cdot y}
% \subseteq \hat X_{\bC^\times \cdot y}$ the induced $\bC^\times$
% Poisson structure, and giving $S$ its induced Poisson structure but
% the new (not necessarily trivial) $\bC^\times$-action.
\subsection{Proof of Theorem \ref{t:pre-main}}\label{ss:t-pre-main-pf}
Here we deduce Theorem \ref{t:pre-main} from the preceding theorems.
Part (i) is an immediate consequence of Theorem
\ref{t:main-sl}.(i). Part (ii) is an immediate consequence of Theorem
\ref{t:e-db} and Theorem \ref{t:main-sl}.(ii). Part (iii) follows from
Corollary \ref{c:ham-pbv-isom} and Theorem \ref{t:main-sr}. Part (iv) 
is identical to Corollary \ref{c:ct-act}.

\subsection{Proof of Theorem \ref{t:main-sl}}\label{ss:t-main-sl-pf}
We first point out that, if the reader has not seen the proof in the
simpler symplectic setting in \S \ref{ss:without-moser}, it may be
helpful to read that first, as the proof of (ii) below follows the
same ideas but is a bit more complicated.

Part (i) of Theorem \ref{t:main-sl} is an immediate consequence of
Corollary \ref{c:fn}, together with the fact that the formal
neighborhood of any point in a smooth irreducible variety of dimension
$n$ is isomorphic to $\Delta^n$.

Part (ii) is the main part of the proof. We will apply \cite[Lemma
3.2]{Kalss}, following the proof of \cite[Proposition 3.3]{Kalss}.
Write $\hat Y_{\bC^\times \cdot y} = (\bC^\times \times \Delta) \times
\Delta^{\dim Y-2}$ with its standard symplectic structure
\eqref{e:ss-deg-k}.  Let $t,u$ denote homogeneous lifts to $\hat
X_{\bC^\times \cdot y}$ of the corresponding coordinate functions on
$\hat Y_{\bC^\times \cdot y}$ (we actually only require that $t$ is
invertible of weight one and $\{t,u\}(y)=1$, so not the full strength
of Theorem \ref{t:e-db}).  Let $Z := \Spf \bC[t,t^{-1}][\![u]\!]$, so
that we have the projection $X \to Z$ given by the inclusion
$\bC[t,t^{-1}][\![u]\!]$.
% Let $J \subseteq \hat \cO_{X,\bC^\times \cdot y}$ be the ideal of $Z
% := (\bC^\times \times \Delta) \times \{0\} \subseteq \hat
% Y_{\bC^\times \cdot y}$. Then the composition $Z \into X \to Z$ is
% the identity.

Consider the subalgebra $A \subseteq \hat \cO_{X, \bC^\times \cdot y}$
of degree-zero elements commuting with $t$. Let $X' := \Spf A$, so
that there is a projection $X \to X'$. Note that, although $A$ is not
a Poisson subalgebra, $A[t,t^{-1}]$ is, since it consists of the
$\bC^\times$-locally finite elements which commute with $t$.  Since
$\{-,-\}$ has degree $-k$, we conclude that $\{-,-\}':=t^k\{-,-\}$
restricts to a Poisson bracket on $X'$.

We claim that, as $\bC^\times$-formal schemes, $\hat X_{\bC^\times
  \cdot y} \cong Z \times X'$ using the two projections.  For this we
apply \cite[Lemma 3.2]{Kalss}, following the argument of
\cite[Proposition 3.3]{Kalss}.  View $\hat \cO_{X, \bC^\times \cdot
  y}^{\bC^\times}$ as a $\bC[\![u]\!] = \cO_{\Delta}$-module.  Let $J
\subseteq \hat \cO_{X, \bC^\times \cdot y}^{\bC^\times}$ be the
degree-zero part of the ideal in $\hat \cO_{X, \bC^\times \cdot y}$ of
$\bC^\times \cdot y$; note that $J$ defines the topology on
$\cO_{X,\bC^\times \cdot y}^{\bC^\times}$.  Consider the filtration
$J^\bullet$ on $\cO_{X, \bC^\times \cdot y}^{\bC^\times}$ given by
powers of $J$. Note that $J^m$ has finite codimension for all $m$
(since $J$ is a defining ideal for the topology on the affine formal
scheme $\cO_{X,\bC^\times \cdot y}^{\bC^\times}$ which has a unique
closed point, or this can be seen directly).  Note also that $u J^m
\subseteq J^{m+1}$ for all $m$.  Note that $t^{k-1}\{t,u\} \equiv 1
\pmod J$, hence $\{t,u\}$ is invertible.  Consider the operator
$\nabla_u$ defined by $\nabla_u(f) := \{t,u\}^{-1}\{t,f\}$.  This
satisfies $[\nabla_u,u]=0$, hence defines a flat connection on $\hat
\cO_{X, \bC^\times \cdot y}^{\bC^\times}$ as a
$\bC[\![u]\!]$-module. Therefore the assumptions of \cite[Lemma
3.2]{Kalss} are satisfied, and the lemma yields that $\hat \cO_{X,
  \bC^\times \cdot y}^{\bC^\times} \cong \bC[\![u]\!] \hat \otimes
(\hat \cO_{X, \bC^\times \cdot y}^{\bC^\times})^\nabla$.  Note that
the second factor is the subalgebra of degree-zero elements commuting
with $t$.  Since $\hat \cO_{X, \bC^\times \cdot y}$ is the completion
of $\hat \cO_{X, \bC^\times \cdot y}^{\bC^\times}[t,t^{-1}]$, we
obtain also that
\[
\hat \cO_{X, \bC^\times \cdot y} \cong \bC[t,t^{-1}][\![u]\!] \hat
\otimes ((\hat \cO_{X, \bC^\times \cdot y})^\nabla)^{\bC^\times},
\]
i.e., $\hat X_{\bC^\times \cdot y} \cong Z \times X'$ with the
projections as claimed.

Next, applying the formal Darboux-Weinstein theorem to $X'$, we obtain
a Poisson isomorphism $X' \cong \Delta^{\dim Y - 2} \times S$ for some
formal Poisson scheme $S$. Therefore, we get a decomposition of
$\bC^\times$-formal schemes $\hat X_{\bC^\times \cdot y} \cong Z
\times \Delta^{\dim Y - 2} \times S$, which is not Poisson, but for
which $t^k\{-,-\}$ restricts to a Poisson bracket on $\cO(S)$ (as well
as to the standard symplectic structure on $\cO(\Delta^{\dim Y -
  2})$).

However, it is unfortunately not quite true that
the projection $X_{\bC^\times \cdot y} \to Z \times \Delta^{\dim Y -
  2}$ is Poisson.  In terms of the standard coordinates
$t,u,z_1,\ldots,z_{\dim Y - 2}$, the Poisson brackets are all those of
the standard symplectic structure \eqref{e:ss-deg-k} except for those
of $u$. To fix this, we need to change the coordinate $u$ so that
$\{u,z_i'\} = 0$ for all $i$, with $z_i'$ defined just after
\eqref{e:ss-deg-k}, and also so that $\{t,u\}=t^{1-k}$. 

We explicitly construct in this paragraph a change of coordinates $u
\mapsto u'$ (fixing the other coordinates) so that $\{u',z_i'\}=0$ for
all $i$; an alternative (but nonexplicit) construction using
\cite[Lemma 3.2]{Kalss} is given in the next paragraph.  Let us use
the notation $\int f dz_j$ for the antiderivative of $f \in
\bC[t,t^{-1}][\![u,z_1,\ldots,z_{\dim Y - 2}]\!]$ which is a multiple
of $z_j$, which we extend $\cO(S)$-linearly to $\hat \cO_{X,\bC^\times
  \cdot y}$.
%  Furthermore,
% let $\int f dz_j'$ equal $\int f dz_j$ when $j$ is even and $t^k \int
% f dz_j$ when $j$ is odd (i.e., $\int f dz_j' = (z_j'/z_j) \int f
% dz_j$).  
Let $\tilde J$ be the ideal of $\bC^\times \cdot y$, i.e., $\tilde J$
is generated by $u,z_1,\ldots,z_{\dim Y - 2}$, and the augmentation
ideal of $\cO(S)$.  We then have that $\{u,z_i'\} \in \tilde J$ for
all $i$ by construction (since it is homogeneous and
$\{u,z_i'\}(y)=0$).  Suppose inductively that $\{u,z_i'\} \in (\tilde
J)^m$ for some $m \geq 1$.  Then if $i$ is odd, we can make the change
of coordinates $u \mapsto u':=u + \int \{u,z_i'\}dz_{i+1}$, which is
congruent to $u$ modulo $(\tilde J)^{m+2}$, and satisfies $\{u',z_i'\}
\in (\tilde J)^{m+1}$. Moreover, if we had in fact $\{u,z_j'\}=0$ for
some $j$, then we claim that still $\{u',z_j'\}=0$:
\[
\{u',z_j'\}=\{\int \{u,z_i'\} dz_{i+1}, z_j'\} = \int \{u, z_j'\}
dz_{i+1} = 0,
\]
where for the second equality we use that $\{\int f dz_{i+1}, z_j'\} =
\int \{f, z_j'\} dz_{i+1}$ for all $f$, which follows since
$\{z_{i+1}, z_j'\}$ is a power of $t$.  Thus, iteratively performing
the change of coordinates $u \mapsto u'$ we get a convergent change of
coordinates after which $u$ commutes with $z_i'$ without affecting the
property of commuting with other $z_j'$.  We can do the same procedure
as above if $i$ is even, except making the change of coordinates $u
\mapsto u' := u - t^k \int \{u,z_i'\} dz_{i-1}$ (note $z_i'=z_i$ in
this case).  Thus we have explicitly shown how to change coordinates
so that $u$ commutes with $z_i'$ for all $i$.

Alternatively, we could have again appealed to \cite[Lemma
3.2]{Kalss}: consider again $\hat \cO_{X,\bC^\times \cdot
  y}^{\bC^\times}$ filtered by powers of $J$, viewed now as a
$\bC[\![z_1,\ldots,z_{\dim Y - 2}]\!]$-module. Define the operators
$\nabla_{z_i} := \{z_i',-\}$ for $i$ odd and $\nabla_{z_i} := t^k
\{z_i,-\}=t^k\{z_i',-\}$ for $i$ even. This defines a connection since
$\nabla_{z_i}(z_j f) = z_j \nabla_{z_i}(f) + \delta_{ij} f$ and it is
flat since the $\nabla_{z_i}$ commute.  Therefore \cite[Lemma
3.2]{Kalss} implies that $\hat \cO_{X,\bC^\times \cdot
  y}^{\bC^\times} \cong \bC[\![z_1,\ldots,z_{\dim Y - 2}]\!] \hat
\otimes (\hat \cO_{X,\bC^\times \cdot y}^{\bC^\times})^\nabla$, and
the latter factor is the centralizer of the $z_i'$.  There
must be an element of the latter factor, call it $u'$, such that
$\{t,u'\}$ does not vanish along $\bC^\times \cdot y$ (as there is no
such element in the first factor). The change of coordinates $u
\mapsto u'$ will then have the desired property.

Finally, we change coordinates so that $\{t,u\} = t^{1-k}$, using an
explicit algorithm similar to the case $\{u,z_i'\}=0$.  Since $u$ and
$t$ now commute with $z_i'$ for all $i$, so must $\{t,u\}$, so that
\[
\{t,u\} \in t^{1-k}(1 + (\bC[\![u]\!] \hat \otimes \cO(S))_+),
\]
with the subscript of $+$ denoting the augmentation ideal.  Suppose
inductively that $\{t,u\} - t^{1-k} \in (\bC[\![u]\!] \hat \otimes
\cO(S)_+)^m$ for some $m \geq 1$ (which is guaranteed for $m=1$).
Then, we perform the change of coordinates $u \mapsto u'=u+\int
t^{k-1} (1-\{t,u\}) du$, which will have the property that $\{t,u'\} -
t^{1-k} \in (\bC[\![u]\!] \hat \otimes \cO(S)_+)^{m+1}$, completing
the induction step.  Note that these changes of coordinates will not
affect the property $\{u,z_i'\}=0$. 

After this, we have that the projection $\hat X_{\bC^\times \cdot y}
\to Z \times \Delta^{\dim Y -2}$ is Poisson, as desired.  To conclude,
note that the composition $\hat Y_{\bC^\times \cdot y} \to \hat
X_{\bC^\times \cdot y} \to Z \times \Delta^{\dim Y - 2}$ of the
inclusion with the projection to the first two factors is
$\bC^\times$-equivariant, sends $y$ to $((1,0),0) \in Z \times
\Delta^{\dim Y-2}$ and is an isomorphism on the tangent space. By
equivariance, it is an isomorphism at the tangent spaces of all closed
points and a bijection of closed points, hence it is an
isomorphism. As it is a composition of Poisson morphisms, this
isomorphism is Poisson.  We can therefore compose with the inverse of
this isomorphism to view our decomposition as $\hat X_{\bC^\times
  \cdot y} \cong \hat Y_{\bC^\times \cdot y} \times S$, with the first
projection Poisson. Finally, the Poisson structure
$\{-,-\}'=t^k\{-,-\}$ on $X'$ induces a Poisson structure on $S$ via
the projection $X' \to S$, also given by the same formula.

We now prove (iii). By part (ii), we can write the Poisson bivector on
$\hat X_{\bC^\times \cdot y} \cong \hat Y_{\bC^\times \cdot y} \times S$ as
\[
\pi_{\hat X_{\bC^\times \cdot y}} = \pi_{\hat Y_{\bC^\times \cdot y}} + t^k \pi_S + \pi',
\]
where $\pi'$ is a bivector which projects to zero on each factor. As
in (ii), set $Z = \bC^\times \times \Delta$ and write $\hat
Y_{\bC^\times \cdot y} \cong Z \times \Delta^{\dim Y - 2}$ with the
standard symplectic structure \eqref{e:ss-deg-k}.  By construction,
$\{t,\cO(S)\}$ and $\{\cO(\Delta^{\dim Y - 2}),\cO(S)\}$ are zero.  Therefore
$\pi'$ has the form $\pi' = t^{-k} \partial_u \wedge \xi$ where
$\xi(\hat \cO_{Y,\bC^\times \cdot y})=0$, that is, $\xi$ is constant
in the $\hat Y_{\bC^\times \cdot y}$ direction, and $\xi$ has degree
zero. To conclude, it suffices to show that $\xi$ is actually parallel
to $S$, i.e., $\xi(\cO(S)) \subseteq \cO(S)$.  Note that $\xi = t^{k}
\xi_u$ where $\xi_u$ is the Hamiltonian vector field of $u$.  But,
$\cO(S)$ is the subalgebra of degree-zero elements commuting with the
coordinate functions $t, z_1, \ldots, z_{\dim Y - 2}$ on $\hat
Y_{\bC^\times \cdot y}$.  To show this is preserved by $\xi$, it
suffices to show that the subalgebra $\bC[t,t^{-1}] \hat \otimes
\cO(S)$ of all elements commuting with $t,z_1,\ldots,z_{\dim Y - 2}$
is preserved by $\xi_u$. This follows since
$\xi_{t}(f)=0=\xi_{z_i}(f)=0$ for all $i$, then $\xi_t(\xi_u(f))=
[\xi_t,\xi_u](f)=\xi_{\{t,u\}}(f) = \xi_{t^{-k-1}}(f)=0$ and
$\xi_{z_i}(\xi_u(f))=[\xi_{z_i},\xi_u](f)=\xi_{\{z_i,u\}}(f)$, which
again is zero since $\{z_i,u\}$ is either zero (if $i$ is even) or
$-kt^{-k}z_i$ (if $i$ is odd). Thus, we obtain \eqref{e:hx-pbv} with
$\xi$ a vector field on $S$.

It remains to verify that $[\xi,\pi_S]=k\pi_S$.  Since $\xi$ is
Hamiltonian on $\hat X_{\bC^\times \cdot y}$, it follows that
$[\xi,\pi]=0$ where $\pi$ is the Poisson bivector on $\hat
X_{\bC^\times \cdot y}$. Since $\pi_S$ is the restriction of $t^k \pi$
to $S$ and $\xi$ is parallel to $S$, we obtain $[\xi,\pi_S] = \xi(t^k)
\pi|_S = kt^k\pi|_S = k \pi_S$.

% For the second paragraph, we proceed as in the proof of Theorem \ref{t:main-sl}.(ii). Note that the centralizer of $t, z_1,
% \ldots, z_{\dim Y -2}$ is a Poisson algebra $\cO_T$ with $t$ a central
% invertible homogeneous coordinate of degree $\ell$, but $T$ need not
% decompose as a product $S \times \bC^\times$. 
% %%%%%ADD MORE
% As in Theorem \ref{t:main-sl}.(ii), we can 
% The above proof then shows that we have a $\bC^\times$-isomorphism
% $\hat X_{\bC^\times \cdot y} \cong \Delta^{\dim Y - 1} \times T$, with
% $\cO(\Delta^{\dim Y - 1}) = \bC[\![u,z_1,\ldots,z_{\dim Y - 2}]\!]$ in
% the above coordinates.  We may rewrite the RHS as $(\Delta^{\dim Y -
%   1} \times \bC^\times) \times_{\bC^\times} T$, reintroducing the $t$
% coordinate into the first factor, yielding \eqref{e:prod-t}. Then, by
% construction, we get a Poisson bivector of the form
% \eqref{e:hx-pbv-t}. Note that now $\hat X_{\bC^\times \cdot y} \to T$
% is an honest Poisson projection, since $\cO(T)$ is closed under the
% Poisson bracket without having to rescale by a power of $t$.
% %%%%%ADD MORE
% For this reason, $\xi=\xi_u$ is now Poisson (in other words, $\xi$ is
% Hamiltonian on $\hat X_{\bC^\times \cdot y}$, so it restricts to a
% Poisson derivation on the Poisson subalgebra $\cO(T) \subseteq \hat
% \cO(X)_{\bC^\times \cdot y}$).

\subsection{Proof of Theorem \ref{t:conv}}
Fix $m$ and $k$, and suppose we are given $(S, \pi_S, \xi)$ and $(S',
\pi_{S'}, \xi')$ satisfying $[\xi,\pi_S] = k \pi_S$ and
$[\xi',\pi_{S'}] = k \pi_{S'}$. Suppose that we have a 
 $\bC^\times$-Poisson isomorphism
\[
\Phi: \bC^\times \times \Delta^{2m-1} \times S \to
\bC^\times \times \Delta^{2m-1} \times S',
\]
equipping both with the Poisson bivectors \eqref{e:hx-pbv} and the
$\bC^\times$-actions by acting on the first factor.  Let $t,u,z_i$ and
$t', u', z_i'$ be the corresponding coordinates on each satisfying
$\{z_{2i-1},z_{2i}\}=1$ (so, the coordinates which were denoted $z_i'$
before, but we have to change notation here due to the primes in the
target of $\Phi$).

First, we claim that we can precompose with an automorphism of the source
such that $\Phi^*(t')=t$.  Write $\Phi^*(t')=tf$ for some degree-zero
function $f$ which is nonvanishing at the ideal $\tilde J$ of the
punctured line $L := \bC^\times \times \{0\} \times \{0\}$.
Therefore $f$ is invertible.  Up to composing with an isomorphism $t
\mapsto \lambda t, u \mapsto \lambda^{-1} u$ for some $\lambda \in
\bC^\times$ (which is the identity on $z_i$ and on $S$) we can
assume that $f \in 1 + \tilde J$.  Set $f_1 := f$. We prove by
induction that there is a sequence $(f_1,f_2,\ldots)$ of elements $f_q
\in 1 + (\tilde J)^q$ and a sequence of  $\bC^\times$-Poisson automorphisms
$(\Psi_1,\Psi_2,\ldots)$ such that $\Psi_q^*$ is the identity modulo
$(\tilde J)^{q}$ and takes $tf_q$ to $tf_{q+1}$. Then the
composition $\cdots \Psi_2^* \Psi_1^* \Phi^*$ converges to a 
graded continuous Poisson algebra isomorphism taking $t'$ to $t$, proving the
claim.

Since continuous Poisson algebra automorphisms which are the identity modulo $(\tilde
J)^{q}$ are all exponentials of Poisson vector fields which are zero
modulo $(\tilde J)^{q}$, the statement is equivalent to finding
Poisson vector fields $\xi_q$ which are zero modulo $(\tilde J)^{q}$
such that
\[
\xi(tf_q) + tf_q \equiv t \pmod{(\tilde J)^{q+1}}.
\]
In fact, we have Hamiltonian vector fields with this property, by
the identity
\[
\{t^k\int (f_q-1) du, tf_q\} \equiv t(1-f_q)f_q \equiv t - t f_q
\pmod{(\tilde J)^{q+1}},
\]
where $\int (f_q-1) du$ is the unique antiderivative of $(f_q-1)$ with
respect to $u$ which is a multiple of $u$ (note that
$\int(f_q-1)du \in (\tilde J)^{q+1}$.)

Next, we may suppose that the Poisson bivectors of $S$ and $S'$ vanish
at the origin; otherwise we could decompose $S$ and $S'$ as products
of $\Delta^{2r}$ by some formal Poisson schemes with Poisson bivectors
vanishing at the origin, where $2r$ is the rank of the Poisson
bivectors of $S$ and $S'$ at the origin (which must be equal since the
ranks of the Poisson bivectors on the products $\bC^\times \times
\Delta^{2m-1} \times S$ and $\bC^\times \times \Delta^{2m-1}
\times S$ are $2r+2m$).  In this case $\Phi$ must preserve the
tangent spaces to $\bC^\times \times \Delta^{2m-1}$ since these
tangent spaces are the spans of the Hamiltonian vector fields.  Up to
change of coordinates, we can assume that $\Phi$ acts as the identity
on this tangent space (by the proof of Theorem \ref{t:e-db}).  

A standard argument used to show uniqueness of the slices in the usual
formal Darboux-Weinstein decomposition shows that we can apply
Hamiltonian isomorphisms so that the coordinates $z_1',\ldots,z_{2m-2}'$
map to $z_1,\ldots,z_{2m-2}$.  
%%This is similar to Darboux theorem...
Namely, suppose that $\Phi^*(z_k') - z_k \in (\tilde J)^{q+1}$ 
 for $k < 2i-1$ and $\Phi^*(z_k') - z_k \in (\tilde J)^q$ for $k \geq 2i-1$.
Then write
$f:= \Phi^*(z_{2i-1}') - z_{2i-1} \in (\tilde J)^q$. For $k < 2i-1$,
we have
\[
\{f,z_k\} =\{\Phi^*(z_{2i-1}),z_k\} = \{\Phi^*(z_{2i-1}), \Phi^*(z_k)-z_k\}
\in (\tilde J)^q.
\]
%%%%%Do we have enough details here?
Then, we can apply the Hamiltonian isomorphism $\exp(\xi_{\int f dz_{2i}})$,
which is the identity modulo
$(\tilde J)^q$, leaves $\Phi^*(z_k')$ unchanged modulo $(\tilde J)^{q+1}$ for
$k < 2i-1$, and takes $\Phi^*(z_{2i-1})$ to an element equivalent to $z_{2i-1}$
modulo $(\tilde J)^{q+1}$.
Similarly, if then
 $\Phi^*(z_{2i}) = z_{2i}+g$ for $g \in (\tilde J)^q$,
% in the $q$-th power of
%the augmentation ideal (and $\{g,z_{2i-1}\}$ is also in the $q$-th power
%of the augmentation ideal), 
 we can then apply $\exp(\xi_{-\int g dz_{2i-1}})$. After this, we
 will have that $\Phi^*(z_k') - z_k \in (\tilde J)^{q+1}$ for $k <
 2i+1$ and $\Phi^*(z_k')-z_k \in (\tilde J)^q$ for $k \geq 2i+1$,
 completing the induction. Moreover, since $\Phi(t')=t$, we
 had $\{\Phi(z_i'), t\} = 0$, which shows that the above procedure
 will preserve $t$.

We may therefore assume that $\Phi^*(t')=t$ and $\Phi^*(z_i')=z_i$ for
all $i$.  Thus $\Phi^*(u') = u + f$ for some $f$ in the augmentation
ideal of $\cO_{S}$.  It follows that $\cO_{S'}$, which is the
weight-zero part of the centralizer of $t', z_1', \ldots, z_{2m-2}'$,
must map to $\cO_{S}$. Thus $\Phi$ produces a Poisson isomorphism $S
\to S'$, as desired.

Finally if we compose with a continuous graded Poisson algebra
automorphism fixing $t, z_1, \ldots, z_{2m}$ and sending $u$ to $u+f$
(for $f$ in the augmentation ideal), we see that $f$ must actually be
in the augmentation ideal of $\cO_S$. That is, this changes $\xi$
appearing in \eqref{e:hx-pbv} to $\xi + t^{-k}\xi_f$, i.e., $\xi +
\xi'_f$ where $\xi'_f$ is the Hamiltonian vector field on $(S,\pi_S)$
associated to $f$.  Conversely given any $f$, we can apply the
automorphism which is the identity on $t, z_i$, and $\cO_S$, sending
$u$ to $u+f$, which is Poisson from the bivector in \eqref{e:hx-pbv}
to the one replacing $\xi$ with $\xi+\xi'_f$.

% For the more general statement, the same argument applies, setting
% $S = T/\!/\bC^\times$ (i.e., $\cO_T = \cO_S^{\bC^\times}$).

\subsection{Proof of Theorem \ref{t:main-sr}}
For most of the proof, we follow (and generalize) the proof of
Proposition \cite[Proposition 5.5]{PS-pdrhhvnc}.

We first show that the hypothesis of (i) implies (ii). In the more
general situation where $X$ is a symplectic singularity, this is
actually a consequence of \cite{Xu-fafg}, as pointed out in
\cite{Nam-fgss}: there it is shown that, for a small enough neighborhood
$U$ of $y$, the algebraic fundamental group $\hat \pi_1(U^\circ)$ of
the smooth locus $U^\circ$ of $U$ is finite.  This in particular
implies that the fundamental group $\pi_1(U^\circ)$ cannot have $\bZ$
as a quotient, so that $H_1(U^\circ)$ is torsion.  Thus,
$H_1(U^\circ,\bC) = 0$, and hence $H^1(U^\circ, \bC)=0$, as desired.

However, we give a more direct proof (using only \cite[Theorem
2.12]{Kalss}) in the case that $X$ has a (projective) symplectic
resolution. Note that, by \cite[Theorem 2.12]{Kalss}, if $\rho: \tilde
X \to X$ is a symplectic resolution, then $H^1(\rho^{-1}(y),\bC)=0$.
Take some tubular analytic neighborhood of $\rho^{-1}(y)$ which
contracts to $\rho^{-1}(y)$.  Since $\rho^{-1}(y)$ is compact, there
is some open ball $U$ around $y$ such that $\rho^{-1}(U)$ lands in the
tubular neighborhood and therefore contracts to $\rho^{-1}(y)$.  Thus
$H^1(\rho^{-1}(U),\bC)=0$.

Let $U^\circ$ be the smooth locus of $U$, and write $U\setminus
U^\circ = \bigcup_i U_i$ for some closed subsets $U_i \subseteq U$.
Since $\rho$ is an isomorphism over $U^\circ$, it suffices to show
that $H^1(\rho^{-1}(U^\circ),\bC)$.  By the semismallness property
\cite[Lemma 2.11]{Kalss} of $\rho$, the complex codimension of
$\rho^{-1}(U_i)$ is at least half the codimension of $U_i$.  Thus, the
real codimension of $\rho^{-1}(U_i)$ is at least equal to the complex
codimension of $U_i$. Since the fundamental group of a smooth manifold
is unchanged by removing a locus of real codimension greater than two
(as all homotopies can be pushed away from the removed locus), we
conclude that the fundamental group of $\rho^{-1}(U^\circ)$ equals
that of $\rho^{-1}(U) \setminus \sqcup_{\dim_{\bC} U_i = \dim_{\bC} U
  - 2} \rho^{-1}(U_i)$.  Moreover, since we can homotope any path in
$U$ away from a locus of real codimension two, the map
$\pi_1(\rho^{-1}(U^{\circ})) \to \pi_1(\rho^{-1}(U))$ is surjective.
For each $U_i$ of complex codimension two, the singularity at each $y
\in U_i$ is of Kleinian type.  The fundamental group of the smooth
locus of a Kleinian singularity is finite (it is the corresponding
group $\Gamma < \operatorname{SL}(2,\bC)$).  Therefore, the kernel of
$\pi_1(\rho^{-1}(U^{\circ})) \to \pi_1(\rho^{-1}(U))$ is generated by
torsion elements. Thus, $H_1(\rho^{-1}(U), \bZ)$ is a quotient of
$H_1(\rho^{-1}(U^{\circ}),\bZ)$ by torsion, and
$H_1(\rho^{-1}(U^{\circ}), \bC) \to H_1(\rho^{-1}(U), \bC)$ is an
isomorphism.  Since $H_1(\rho^{-1}(U),\bC) = 0$, also
$H_1(\rho^{-1}(U^{\circ}),\bC)=0$, and hence
$H^1(\rho^{-1}(U^{\circ}),\bC)=0$.  Thus $H^1(U^{\circ},\bC)=0$, as
desired.

Next, we explain why part (ii) implies (iii). For this we can assume
$S$ comes from an ordinary Darboux-Weinstein decomposition; otherwise
completing at $y$, we still get a product of formal schemes $\hat X_y
\cong \hat Y_y \times S$ (it is just not necessarily Poisson),
and this is all we will use.  Assume (ii). Let $\rho: V \to U$ be an
arbitrary resolution of singularities (not necessarily symplectic).
Assume $U$ is small enough that $V$ contracts to $\rho^{-1}(y)$.  Let
$Z \subseteq U$ be the singular locus, so $U^\circ = U \setminus Z$.
Let $\bH_{DR}^\bullet$ denote the de Rham hypercohomology.  By
Hartshorne's theorem \cite{Har-adrc,Har-drcav}, $\bH^\bullet_{DR}(\hat
V_{\rho^{-1}(y)}) \cong H^\bullet(\rho^{-1}(U)) \cong
H^\bullet(\rho^{-1}(y))$, the topological cohomology.  Applying the
Mayer-Vietoris sequence (cf.~\cite[(4.40)--(4.44)]{ES-dmlv}), we find
that $\bH^\bullet_{DR}(\hat V_{\rho^{-1}(y)} \setminus \rho^{-1}(Z))
\cong H^\bullet(\rho^{-1}(U) \setminus \rho^{-1}(Z)) \cong
H^{\bullet}(U^\circ)$.  On the other hand, $\hat V_{\rho^{-1}(y)}
\setminus \rho^{-1}(Z) = \hat U_y \setminus Z$, since $\rho$ is an
isomorphism over $U^{\circ} = U \setminus Z$.  Thus we conclude that
$\bH^\bullet_{DR}(\hat U_y \setminus Z) \cong H^\bullet(U^{\circ})$.
Now, $\tilde \Omega^\bullet_{\hat U_y}$ is the global sections of the
algebraic de Rham complex on $\hat U_y \setminus Z$ (since $U
\setminus Z$ is the smooth locus of $U$, $U$ is normal, and $Z$ has
codimension at least two).  Therefore, by the spectral sequence
computing hypercohomology, $H^1(\tilde \Omega^\bullet_{\hat U_y})$ is a
summand of $\bH^1_{DR}(\hat U_y \setminus Z) = H^1(U^{\circ})$
(alternatively, we can see that the canonical map $H^1(\tilde
\Omega^\bullet_{\hat U_y}) \to H^1(U^{\circ})$ is injective since any
algebraic one-form which is the differential of a smooth function is
actually the differential of an algebraic function, and the same
applies after taking formal completions).  By the hypothesis, the
latter is zero, establishing $H^1(\tilde \Omega^\bullet_{\hat U_y}) =
0$. Since $\hat U_y \cong \Delta^{\dim Y} \times S$, we also get
$H^1(\tilde \Omega^\bullet_{S}) = 0$, as desired.

Finally, we explain why part (iii) implies (iv).  Suppose $H^1(\tilde
\Omega^\bullet_{S}) = 0$ and $X$ (hence $S$) is generically
symplectic.  If $\xi$ is a Poisson vector field on $S$, i.e., $[\xi,
\pi] = 0$ for $\pi$ the Poisson structure, then for $\omega =
\pi^{-1}$ the generic symplectic structure, $i_{\xi} \omega$ is a
closed one-form on the smooth locus of $S$, and it is exact if and
only if it is Hamiltonian.

% For the last statement, the same argument shows (ii) implies (iii').
% Assuming (iii'), given any degree-zero Poisson vector field $\xi$ on
% $T$ annihilating $t$, we obtain a Poisson vector field on
% $S=T/\!/\bC^\times$, which is generically symplectic.  Thus we get a
% closed one-form $i_{\xi} \omega$ on $S$, which is by restriction a
% closed one-form on $T$. This is exact, so $i_\xi \omega = df$ for
% some $f \in \cO_T$.  Since $f$ has degree zero, $f \in \cO_S =
% \cO_T^{\bC^\times}$.  Thus the restriction of $\xi$ to $S$ equals
% $\xi_f$. Since $T$ is irreducible, $\xi = \xi_f$ as in the argument
% at the end of the proof of Corollary \ref{c:ham-pbv-isom}.

\subsection{Quantization}\label{ss:quant}
Parallel to Theorem \ref{t:q-unique}, Theorem \ref{t:main-sl} has the
following consequence. Let $\caD_{\hbar,k}(\bC^\times \times
\Delta^{n - 1})$ denote the completion of \eqref{e:gr-quant} with
respect to the ideal $(\hbar,u,z_1',\ldots,z_{2n-2}')$, i.e., the
quantization of $\bC^\times \times \Delta^{2n-1}$ discussed in \S
\ref{ss:quant1}.
\begin{theorem}\label{t:quant}
  Let $X, Y, y$, and $k$ be as in Theorem \ref{t:main-sl}. Then, every $\bC^\times$-compatible quantization $A_\hbar$
  of $\hat X_{\bC^\times \cdot y}$ has subalgebras
  $\caD_{\hbar,k}(\bC^\times \times \Delta^{(\dim Y)/2 - 1})$ and
  $A_\hbar'$ quantizing $\hat Y_{\bC^\times \cdot y}$ and $\bC^\times
  \times S$, respectively, so that $A_\hbar'$ is the simultaneous
  centralizer of $t,z_1,\ldots, z_{\dim Y - 2}$. We have
\begin{equation}\label{e:q-decomp}
  A_\hbar \cong \caD_{\hbar,k}(\bC^\times \times \Delta^{(\dim Y)/2 - 1}) \hat \otimes_{\bC[t,t^{-1}][\![\hbar]\!]} A_\hbar'
\end{equation}
as graded topological $\bC[\![\hbar]\!]$-modules.

In the case that all Poisson vector fields on $S$ are Hamiltonian (or
any of the conditions of Theorem \ref{t:main-sr} are satisfied), then
we can choose the coordinate $u$ so that \eqref{e:q-decomp} is a
graded algebra isomorphism. If, further, the conditions of Corollary
\ref{c:ct-act} are satisfied, i.e., $S$ also admits a $\bC^\times$
action giving its Poisson structure degree $-k$, and this action
quantizes to a compatible action on $A_\hbar'$ (now acting trivially
on $\bC[t,t^{-1}] \subseteq A_\hbar'$), then we can replace the tensor
product above with one over $\bC[\![\hbar]\!]$: for $A_\hbar''
\subseteq A_\hbar'$ the subspace where the two $\bC^\times$ actions
have the same weight,
\begin{equation}\label{e:q-decomp2}
  A_\hbar \cong \caD_{\hbar,k}(\bC^\times \times \Delta^{(\dim Y)/2 - 1}) \hat \otimes_{\bC[\![\hbar]\!]} A_\hbar'',
\end{equation}
as graded algebras.
\end{theorem}
% \begin{remark} 
%   If we don't assume $\ell=1$, then \eqref{e:q-decomp} still holds,
%   where now $A_\hbar'$ quantizes $T$ as in Theorem
%   \ref{t:main-sl}.(iv).  We see in the theorem that a quantization of
%   $X$ always yields a quantization, $A_\hbar'$, of $T$, which
%   indicates the naturality of $T$. In contrast, to get the
%   quantization $A_\hbar''$ of the slice $S$ in our
%   $\bC^\times$-equivariant context, we needed a further assumption
%   that not merely $S$ but also the quantization of $T$ admit a
%   $\bC^\times$ action.
% \end{remark}
Motivated by the additional hypothesis in the second paragraph of
Theorem \ref{t:quant}, we can ask the following quantum analogue of
Question \ref{q:slice}:
\begin{question}\label{q:q-slice}
  Suppose $X$ is conical and admits a $\bC^\times$-equivariant
  symplectic resolution with homogeneous symplectic form of degree $k$
  (or is a symplectic singularity and has a symplectic structure of
  degree $k$ on the smooth locus). Let $x \in X$ have trivial
  stabilizer under $\bC^\times$ and suppose that Question
  \ref{q:slice} has a positive answer, i.e., the $S$ in \eqref{e:ehat}
  admits a contracting $\bC^\times$ action for which the Poisson
  structure is homogeneous, and suppose it has degree $k$.  Then,
%\begin{enumerate}
does every $\bC^\times$-compatible quantization of $X$ contain
a subalgebra which quantizes $S$ and admits a compatible lift of the $\bC^\times$
action on $S$? 
%Does it arise from the construction as in the theorem above?
%\item Does this subalgebra arise from the above, via a compatible lift
%of the action on $T$?
%\end{enumerate}
\end{question}
We see in the next sections that this question has a positive answer,
at least, for certain quantizations of linear quotients and hypertoric
varieties.  Note, as pointed out after the statement of Question
\ref{q:slice}, that the degree $k$ must be positive in Question
\ref{q:q-slice}.
\begin{remark}\label{r:nam2}
  As in Remark \ref{r:nam}, one can also weaken the hypotheses of the
  question and require only that $X$ be normal and conical and have a
  symplectic form of positive degree $k$ on its smooth locus (and not
  require that $X$ be a symplectic singularity).
\end{remark}
\begin{proof}[Proof of Theorem \ref{t:quant}]
  By Theorem \ref{t:main-sl}, it is enough to replace $\hat
  X_{\bC^\times \cdot y}$ with $(\bC^\times \times \Delta^{2n-1})
  \times S$, with Poisson bivector \eqref{e:hx-pbv}.  We moreover view
  the quantization as an associative star product on $\cO(\bC^\times
  \times \Delta^{2n-1} \times S)[\![\hbar]\!]$. In this paragraph, we
  apply arguments similar to those in \S \ref{ss:without-moser} and in
  the proof of Theorem \ref{t:q-unique} in order to find a gauge
  transformation taking $t,u,z_1,\ldots,z_{2n-2}$ to coordinates
  satisfying the desired commutation relations; this will actually be
  surprisingly simple since the commutators are already correct modulo
  $\hbar^2$ (and this will produce a slightly different, more
  explicit, proof of Theorem \ref{t:q-unique}). We continue to use the
  $z_i'$ defined as before, $z_i' = z_i$ for $i$ even and $z_i' = t^k
  z_i$ for $i$ odd).  Let $[a,b]_\star := a \star b - b \star a$.  To
  begin, we know that commutators of these coordinates are correct
  modulo $\hbar^2$. Suppose that they are correct modulo $\hbar^m$ for
  some $m \geq 2$.  Then, we can apply coordinate changes $u \mapsto u
  - \hbar^{-1} t^{k-1} \int ([t,u]_\star-\hbar t^{1-k}) du, z_i
  \mapsto z_i - \hbar^{-1} t^{k-1} \int [t,z_i]_\star du$ for all $i$
  so that the commutators with $t$ are now correct modulo
  $\hbar^{m+1}$.  Similarly, for $i$ even and between $2$ and $2n-2$,
  we can apply coordinate changes $u \mapsto u - \hbar^{-1} t^k \int
  [u,z_i']_\star dz_{i-1}$ and for $i$ odd, $u \mapsto u + \hbar^{-1}
  \int [u,z_i']_\star dz_{i+1}$, so that the commutators with $u$ are
  correct modulo $\hbar^{m+1}$. Finally, with $j > i$, we can apply
  coordinate changes $z_j \mapsto z_j - \hbar^{-1} t^k \int
  [z_j,z_i]_\star dz_{i-1}$ for $i$ even, $z_j \mapsto z_j +
  \hbar^{-1} t^k \int [z_j,z_i]_\star dz_{i+1}$ when $i$ is odd and $j
  \neq i+1$, and finally $z_{2i} \mapsto z_{2i} + \hbar^{-1} t^k \int
  ([z_{2i-1},z_{2i}]_\star dz_{2i}$ for all $i$, after which all
  commutators among the $t,u,z_1, \ldots, z_{2n-2}$ will be correct
  modulo $\hbar^{m+1}$.

  We now assume that the $t,u,z_1,\ldots,z_{2n-2}$ have the correct
  commutation relations. They therefore generate the desired
  subalgebra $\caD_{\hbar,k}(\bC^\times \times \Delta^{(\dim Y)/2 -
    1})$ quantizing $\hat Y_{\bC^\times \cdot y}$.  We define
  $A_\hbar'$ as the subalgebra of elements commuting with $t, z_1,
  \ldots, z_{2n-2}$.  We need to show that \eqref{e:q-decomp} holds.
  To do this, we will show that there is a continuous graded
  $\bC[\![\hbar]\!]$-linear map $\Phi: \cO(S \times
  \bC^\times)[\![\hbar]\!] \to \cO(\bC^\times \times \Delta^{2n-1}
  \times S)[\![\hbar]\!]$, which is the inclusion modulo $\hbar$, such
  that the image of $\Phi$ is $A_\hbar'$. Taking the product of this
  with the inclusion we obtain a continuous $\bC[\![\hbar]\!]$-linear
  map $\cO(\bC^\times \times \Delta^{2n-1} \times S)[\![\hbar]\!] \to
  \cO(\bC^\times \times \Delta^{2n-1} \times S)[\![\hbar]\!]$ which is
  an isomorphism modulo $\hbar$, hence an isomorphism of graded topological
  $\bC[\![\hbar]\!]$-modules. It sends $\cO(\bC^\times \times
  \Delta^{2n-1})[\![\hbar]\!]$ to $\caD_{\hbar,k}(\bC^\times \times
  \Delta^{(\dim Y)/2 - 1})$ and $\cO(S \times
  \bC^\times)[\![\hbar]\!]$ to $A_\hbar'$. Thus, we will conclude
  \eqref{e:q-decomp}.

  We construct $\Phi$ order by order.  For every $m \geq 1$, let
  $(A_\hbar')_m \subseteq \cO(\bC^\times \times \Delta^{2n-1} \times
  S)[\hbar]/(\hbar^{m+1})$ be the subalgebra of elements annihilated
  by the operators $\hbar^{-1} \ad(t)$ and $\hbar^{-1} \ad(z_i)$ for
  all
  $i$. % It then suffices to prove that there exists $\Phi_m: \cO(S \times \bC^\times)[\hbar]/(\hbar^{m+1}) \to \cO(\bC^\times
%   \times \Delta^{2n-1} \times S)[\hbar](\hbar^{m+1})$, which is the inclusion
% modulo $\hbar$, whose image is $(A_\hbar')_m$. We can now prove this by induction with the trivial base case $m=0$.
% Assume therefore that,
Assume that, for some $m \geq 1$, there exists $\Phi_m: \cO(S
\times \bC^\times)[\hbar]/(\hbar^{m+1}) \to \cO(\bC^\times \times
\Delta^{2n-1} \times S)[\hbar]/(\hbar^{m+1})$ with image
$(A_\hbar')_m$, which is the identity modulo $\hbar$.  We must extend
this to a map $\Phi_{m+1}$ with image $(A_\hbar')_{m+1}$.  Note first
that the intersection $(A_\hbar')_{m+1} \cap \hbar^{m+1} 
\cO(\bC^\times \times \Delta^{2n-1} \times S)[\hbar]/(\hbar^{m+2})$ is
$\hbar^{m+1} \cdot \cO(\bC^\times \times S)$, since the condition for
a multiple of $\hbar^{m+1}$ to be annihilated by $\hbar^{-1}\ad(f)$
modulo $\hbar^{m+2}$ is just the condition that the element Poisson
commute with $f$, for all $f \in \cO(\bC^\times \times \Delta^{2n-1}
\times S)$, and we apply this to $f \in \{t,z_1,\ldots,z_{2n-2}\}$.  Therefore, the existence of $\Phi_{m+1}$ is equivalent to the statement that every element
$a \in (A_\hbar')_m$ has a lift $\tilde a \in (A_\hbar')_{m+1}$ to an element of
$(A_\hbar')_{m+1}$. This in turn is equivalent to the existence of a solution
$a_{m+1} \in  \bC(\bC^\times \times \Delta^{2n-1} \times S)$ to
the equations $\hbar^{-1}[t,\hbar^{m+1}a_{m+1}]_\star \equiv \hbar^{-1}[t,a]_\star \pmod {\hbar^{m+2}}$ and $\hbar^{-1}[z_i',\hbar^{m+1}a_{m+1}]_\star
\equiv \hbar^{-1}[z_i',a]_\star \pmod{\hbar^{m+2}}$ for all $i$.  

Let $\ad_\star(f)$ be the operator $\ad_\star(f)(g)=[f,g]_\star$.
Define the operators $\nabla_u := \hbar^{-1}t^{k-1}\ad_\star(t)$,
$\nabla_{z_i} := -\hbar^{-1} t^{k} \ad_\star(z_{i+1})$ for $i$ odd,
and $\nabla_{z_i} := \hbar^{-1} \ad_\star(z_{i-1}')$ for $i$ even.
These operators satisfy the identities $\nabla_u(u \star f)=u \star
\nabla_u(f) +f$ and $\nabla_{z_i}(z_i \star f)=z_i\star \nabla_{z_i}f
+ f$. Moreover, the operators commute, by the Jacobi identity for
commutators.  The above equations are equivalent to
\[
\nabla_u(a_{m+1}) \equiv \hbar^{-m-1}\nabla_u(a) \pmod{\hbar}, \nabla_{z_i}(a_{m+1}) \equiv \hbar^{-m-1}\nabla_{z_i}(a) \pmod{\hbar} \forall i.
\]
Modulo $\hbar$, the operators $\nabla_u, \nabla_{z_i}$  define a flat connection
on $\bC(\bC^\times \times
  \times \Delta^{2n-1} \times S)$
over $\bC[\![u,z_1,\ldots,z_{2n-1}]\!] = \bC(\Delta^{2n-1})$. Writing the above equations as $\nabla_u(a_{m+1}) \equiv F_u \pmod{\hbar}$ and $\nabla_{z_i}(a_{m+1})=F_{z_i}\pmod{\hbar}$, it follows from the previous identities that the integrability conditions $\nabla_{z_i} F_u = \nabla_u F_{z_i}$ and $\nabla_{z_i} F_{z_j} = \nabla_{z_j} F_{z_i}$ hold.  Therefore there exist solutions to the above equations.  This completes the proof of \eqref{e:q-decomp}.

  In the case that all Poisson vector fields on $S$ are Hamiltonian,
  we can inductively replace $u$ by a new coordinate function (which
  remains unchanged modulo $\hbar$), so that $u$ commutes with
  $A_\hbar'$ and continues to satisfy the needed commutation relations
  in $\caD_{\hbar,k}(\bC^\times \times \Delta^{(\dim Y)/2 -
    1})$. We begin, by Corollary \ref{c:ct-act}, knowing $[u,A_\hbar']
  \subseteq \hbar^2 A_\hbar$.  Assume $[u,A_\hbar'] \subseteq \hbar^N
  A_\hbar$ for some $N \geq 2$.  Then we can consider the vector field
  $\hbar^{-N} t^{kN} \ad(u)$ on $\cO_S$ valued in $A_\hbar$. In fact,
  it must be valued in the centralizer of $t,u,z_1', \ldots,
  z_{2n-2}'$, by the Jacobi identity, so it is valued in $A_\hbar'$.
  Then, it is Hamiltonian, so we can subtract an element of
  $\hbar^{N-1} A_\hbar'$ from $u$ so that $[u,A_\hbar']\subseteq
  \hbar^{N+1} A_\hbar$.

  Finally, if the assumptions of Corollary \ref{c:ct-act} are
  satisfied and $A_\hbar'$ has a $\bC^\times$-compatible action for
  the new grading (giving $t$ now degree zero), then it remains to
  note that $A_\hbar''$ as defined in the theorem is a graded
  subalgebra of $A_\hbar'$, and that $A_\hbar' = A_\hbar''[t,t^{-1}]$.
\end{proof}

\section{Finite linear quotients}\label{s:flq}
In this section we give more explicit and stronger versions of the
main theorem when $X$ is a finite quotient of a symplectic vector
space. In particular, we explain how one can replace formal
localization by an explicit \'etale (or Zariski) localization.

Let $V$ be a symplectic vector space and $\Gamma < \Sp(V)$ a finite
subgroup. We consider the quotient $X := V/\Gamma$.
% , with quotient map $q: V \onto X$.
The symplectic structure on $V$ makes $X$ a Poisson
variety.  Equip $V$ with the dilation action of $\bC^\times$, which makes
its symplectic form have weight $2$, and hence the Poisson structures
on $V$ and $X$ have weight $-2$.

Recall that a subgroup $\Gamma_0 < \Gamma$ is called \emph{parabolic}
if there exists $v \in V$ such that $\Stab_\Gamma(v) = \Gamma_0$. Let
$V^{\Gamma_0}$ be the fixed point set of $\Gamma_0$, $N(\Gamma_0) <
\Gamma$ be the normalizer, and $N(\Gamma_0)^0 :=
N(\Gamma_0)/\Gamma_0$, the residual action on $V^{\Gamma_0}$.  Then
one has the symplectic leaf $X_{\Gamma_0}:=(V^{\Gamma_0})^\circ /
N(\Gamma_0)^0$, where $(V^{\Gamma_0})^\circ := \{w \in V \mid
\Stab_\Gamma(w)=\Gamma_0\}$. Its closure is $\bar X_{\Gamma_0} =
V^{\Gamma_0} / N(\Gamma_0)^0$.  Moreover, all symplectic leaves $X_{\Gamma_0}$ are obtained in this way, and this
establishes a bijection between conjugacy classes of parabolic
subgroups of $\Gamma$ and  symplectic leaves (or symplectic leaf
closures) of $X$.

Let $\Gamma_0 < \Gamma$ be a parabolic subgroup. Fix $v \in
V^{\Gamma_0}$ with stabilizer equal to $\Gamma_0$, and let $\bar v \in
X_{\Gamma_0}$ be its image in $X$.  We assume (as in Theorem
\ref{t:pre-main}) that $\bC^\times$ acts freely on $\bar v$, i.e.,
that $v$ is not an eigenvector of any nontrivial element of
$N(\Gamma_0)^0$.\footnote{Note that such a $v \in V^{\Gamma_0}$ always
  exists unless $-\Id \in N(\Gamma_0)^0$, since no multiples of the
  identity other than $\pm \Id$ preserve the symplectic form on
  $V^{\Gamma_0}$. In other words, the $\ell$ appearing in Theorem \ref{t:main-sl}
  must be either two or one, and it is two if and only if $-\Id \in N(\Gamma_0)^0$.}
\begin{proposition}\label{p:linquot}
There is a canonical $\bC^\times$-equivariant Poisson isomorphism
\begin{equation}
\hat X_{\bC^\times \cdot \bar v} \cong
 (\widehat {X_{\Gamma_0}})_{\bC^\times \cdot \bar v} \times 
\widehat{(V^{\Gamma_0})^{\perp}}/\Gamma_0.
\end{equation}
\end{proposition}
This implies that Question \ref{q:slice} has an affirmative answer in this case.

The proposition follows easily from the following \'etale local statement.
Let $V^\circ := \{w \in V \mid \Stab_\Gamma(w) < \Gamma_0\}$ and
denote by $X^\circ=V^\circ/\Gamma$ its image.
Then we have an \'etale cover $V^\circ/\Gamma_0 \to X^\circ$.
% Similarly, let $X_{\Gamma_0}^\circ := X_{\Gamma_0} \cap X^\circ
% \cong V^{\Gamma_0,\circ}/N(\Gamma_0)^0$, where $V^{\Gamma_0,\circ}
% := V^{\Gamma_0} \cap V^\circ$.
Note that $X_{\Gamma_0}$ is closed in $X^\circ$.
We have an \'etale covering $(V^{\Gamma_0})^\circ \to X_{\Gamma_0}$. Clearly, passing to the \'etale covers,
\[
V^\circ/\Gamma_0 \cong (V^{\Gamma_0})^\circ \times (V^{\Gamma_0})^{\perp}/\Gamma_0.
\]
%where $V^{\Gamma_0,\circ} = V^\Gamma \cap V^\circ$. 
We note that this holds for all $v \in (V^{\Gamma_0})^\circ$.  Then,
the proposition follows from this together with the fact that, when
$\bC^\times$ acts freely on $\bar v$, then the natural map
$\widehat{V^{\Gamma_0}}_{\bC^\times \cdot v} \to (\widehat
{X_{\Gamma_0}})_{\bC^\times \cdot v}$ of completions is an isomorphism
(since $\bC^\times \cdot v$ does not intersect its image under any
nontrivial element of $N(\Gamma_0)^0$).

\begin{remark} Putting the above together with the results of Section
  \ref{s:main}, we can relate the Darboux-Weinstein decompositions for
  an arbitrary $\bC^\times$-Poisson variety $X$ with a quotient
  $X/\Gamma$ for $\Gamma$ a finite subgroup. Namely, suppose we are
  given $y$ in a symplectic leaf $Y \subseteq X$ which we assume to be
  $\bC^\times$-stable and a $\bC^\times$-equivariant Poisson
  decomposition $\hat X_{\bC^\times \cdot y} \cong \hat Y_{\bC^\times
    \cdot y} \times S$.  As before, let
  $\Gamma_0:=\Stab_\Gamma(y)$, $N(\Gamma_0) < \Gamma$ be the
  normalizer, and $N(\Gamma_0)^0 := N(\Gamma_0)/\Gamma_0$.  Let $z :=
  \Gamma \cdot y \in X/\Gamma$.  Then the 
  symplectic leaf of $z$ in $X/\Gamma$ is $Z
  := Y_{\Gamma_0}/N(\Gamma_0)^0$ for $Y_{\Gamma_0} := \{y' \in Y \mid
  \Stab_\Gamma(y')=\Gamma_0\}$. Assume that $\bC^\times$ acts freely
  on $z$. Then, we then get a decomposition
  $\widehat{X/\Gamma}_{\bC^\times \cdot z} \cong \hat Z_{\bC^\times
    \cdot z} \times \bigl((T_y Y_{\Gamma_0})^\perp \times
  S\bigr)/\Gamma_0$, for $(T_y Y_{\Gamma_0})^\perp \subseteq T_y Y$
  the perpendicular to $T_y Y_{\Gamma_0}$.
\end{remark}

\begin{remark}
In the case above where $V = T^* U$ for $U$ a complex vector space and
$\Gamma < \GL(U) < \Sp(V)$, we can strengthen the proposition by completing only
in the $U$ direction. For $u \in U$ with stabilizer $\Gamma_0$,
and $\bar u \in U/\Gamma \subseteq X$ its image, this yields
% localizing only on $U$.  Let $u \in U^{\Gamma_0}$ have stabilizer
% $\Gamma_0$ and let $p: U \to U/\Gamma$ be the quotient. Then we have
% the \'etale local decomposition,
% \begin{equation}
% T^*(U^{\Gamma_0,\circ}) \times T^*((U^{\Gamma_0})^\perp)/\Gamma_0 \iso
% X^\circ \times_{U^\circ/\Gamma} U^{\circ}/\Gamma_0
% \end{equation}
% which induces the formal decomposition
\begin{multline}
X \times_{U/\Gamma} \widehat{U/\Gamma}_{\bC^\times \cdot \bar u} \cong
 \\ \bigl(X_{\Gamma_0} \times_{U^{\Gamma_0}/N(\Gamma_0)^0} 
\widehat {U^{\Gamma_0}/N(\Gamma_0)^0}_{\bC^\times \cdot \bar u}\bigr) 
\times \bigr(T^*(U^{\Gamma_0})^{\perp}/\Gamma_0 
\times_{(U^{\Gamma_0})^{\perp}/\Gamma_0} \widehat{(U^{\Gamma_0})^\perp}/\Gamma_0\bigl).
\end{multline}
Also, in the case that $N(\Gamma_0)^0 = \{1\}$, then we can replace
all completions along the punctured line $\bC^\times \cdot v$ by
completions along the entire leaf $X_{\Gamma_0}$ (more generally, we
can complete not merely along a punctured line but along any
$\bC^\times$-invariant locally closed subset which does not intersect
its images under $N(\Gamma_0)^0$.)
\end{remark}
\subsection{Quantization}
A standard quantization of $V$ is the algebra of differential
operators on a Lagrangian $U \subseteq V$ (so that $V \cong T^* U$).
So $X$ is quantized by $\caD_\hbar(U)^\Gamma$.  Call this $A_X$, and
similarly define $A_{X_{\Gamma_0}}$ and $A_{V^{\Gamma_0}/\Gamma_0}$.
By a straightforward quantum generalization of Proposition
\ref{p:linquot}, one obtains, under the same assumptions:
\begin{corollary} There is a canonical graded continuous $\bC[\![\hbar]\!]$-algebra
isomorphism
\begin{equation}
  (\hat A_{X})_{\bC^\times \cdot \bar v} \cong
 (\hat A_{X_{\Gamma_0}})_{\bC^\times \cdot \bar v} 
\hat \otimes_{\bC[\![\hbar]\!]} 
A_{V^{\Gamma_0}/\Gamma_0}.
\end{equation}
\end{corollary}
This implies a positive answer to Question \ref{q:q-slice} for these quantizations.

There is a well-known family of quantizations which generalizes the
invariant differential operators (and in fact yields the universal
deformation of these), called spherical symplectic reflection
algebras \cite{EGsra}.  In order to be consistent with typical
notation for symplectic reflection algebras, we will actually consider
the Poisson variety $V^*/\Gamma$, so that its algebra of functions is
$\Sym(V)^\Gamma$.

Recall that one constructs the universal deformation of $\Weyl(V)
\rtimes \Gamma$ as follows (\cite{EGsra},
cf.~e.g.,\cite{Los-csra}). Let $S \subseteq \Gamma$ be the subset of
symplectic reflections, i.e., $s \in S$ if $\rk(s-\Id) = 2$ (this is
the minimal possible nonzero rank since the determinant of $s$ is one;
geometrically, $s$ can be thought of as a generalized reflection
around the symplectic subspace $\ker(s-\Id)$).  Let $S = S_1 \sqcup
S_2 \sqcup \cdots \sqcup S_r$ be the partition into conjugacy classes.
Define the vector space $\mathfrak{c}$ with basis $\hbar, c_1, \ldots,
c_r$, and for $s \in S_i$, we set $c(s) := c_i$.  (Note that
$\mathfrak{c} = \bC \cdot \hbar \oplus \Hom_\Gamma(S,\bC)^*$.)  Then
we define the algebra $H(V,\Gamma)$ as the quotient of the
skew-product algebra $\Sym(\mathfrak{c}) \otimes TV \rtimes \Gamma$ by
the relations
\begin{equation}
[x,y]=\hbar \omega(x,y) + \sum_{s \in S} c(s) \omega_s(x,y) s,
\end{equation}
where $\omega$ is the symplectic form, and $\omega_s$ is the
projection of $\omega$ to $\wedge^2 ker(s-\Id)^\perp$, i.e.,
$\omega_s(x,y) = 0$ if $x$ or $y$ is in $\ker(s-\Id)$, and
$\omega_s(x,y)=\omega(x,y)$ if $x,y \in \img(s-\Id)$.

Then, the (universal) spherical symplectic reflection algebra is
defined as $U(V,\Gamma) := e H(V,\Gamma) e$, where $e :=
\frac{1}{|\Gamma|} \sum_{\gamma \in \Gamma} \gamma \in \bC[\Gamma]$ is
the symmetrizer.
The algebra $U(V,\Gamma)$ encodes the family of
quantizations $\hat U_\lambda(V,\Gamma)$ for
$\lambda=(\lambda_1,\ldots,\lambda_r) \in \bC^r$, where
$U_\lambda(V,\Gamma) := U(V,\Gamma)/(c_i-\lambda_i\hbar)$ and $\hat
U_\lambda(V,\Gamma)$ is its $\hbar$-adic completion. 
 There are canonical quotients $H(V,\Gamma)
\onto \cO(V^*) \rtimes \Gamma$ and $U(V,\Gamma) \onto
\cO(V^*)^\Gamma$, with kernel the ideal generated by
$\mathfrak{c}$. 

Then, Proposition \ref{p:linquot} quantizes to the following
statement.  We will use the projection $\pi: H(V,\Gamma) \to \cO(V^*)
\rtimes \Gamma$ and similarly $\pi: U(V,\Gamma) \to \cO(V^*)^\Gamma$.
Given any closed subvariety $Z \subseteq V^*/\Gamma$, with ideal
$I_Z$, let $\hat H(V,\Gamma)_Z$ denote the completion of $H(V,\Gamma)$
along the (two-sided) ideal generated by $\pi^{-1}(I_Z)$, i.e.,
$\pi^{-1}(\cO(V^*) I_Z \rtimes \Gamma)$.  Define in the same way $\hat
U(V,\Gamma)_Z$, which equals $e\hat H(V,\Gamma)_Z e$.  Given an open
affine subvariety $Y \subseteq Z$, write $Z$ as the complement of the
vanishing of a function $f \in \cO(V^*)^\Gamma$. We can define the
localizations $H(V,\Gamma)[\pi^{-1}(f)^{-1}]$ and
$U(V,\Gamma)[\pi^{-1}(f)^{-1}]$, obtained by inverting all elements in
$\pi^{-1}(f)$ (these are Ore localizations with respect to the set
$\bigcup_{m \geq 1} \pi^{-1}(f^m)$). Then we define $\hat
H(V,\Gamma)_Y$ and $\hat U(V,\Gamma)$ as the completions of these
localizations with respect to the ideals generated by $\pi^{-1}(I_Z)$.

Let $x \in V^*$ be a point whose stabilizer is $\Gamma_0 <
\Gamma$. Let $\bar x \in V^*/\Gamma$ be its image.
% $q: V^* \to V^*/\Gamma$ be the projection, as before, so we have
% $q(x) \in V^*/\Gamma$.

In the above context (which is significantly more delicate than the
context merely of invariant differential operators, due to the
deformed relations), Losev proved a decomposition theorem
\cite[Theorem 1.2.1]{Los-csra}.  Using the result of the previous
section, we can obtain a graded version of his theorem.  Assume as
before that $\bC^\times$ acts freely on $\bar x$.  Suppose that (up to
reordering) $S_1, \ldots, S_i$ are the conjugacy classes of symplectic
reflections which intersect $\Gamma_0$, and let $S_1^0,\ldots,S_i^0$
be their intersections with $\Gamma_0$. We then consider
$U((V^{\Gamma_0})^\perp,\Gamma_0)$ to be a
$\bC[\![\hbar,c_1,\ldots,c_i]\!]$-algebra, defined using the conjugacy
classes $S_1^0,\ldots, S_i^0$.
\begin{theorem}\label{t:sra-gc}
  There are graded continuous $\bC[\![\hbar,c_1,\ldots,c_r]\!]$-algebra isomorphisms
\begin{gather}
  \hat U(V,\Gamma)_{\bC^\times \cdot \bar x} \cong \hat
  U(V^{\Gamma_0},\{1\})_{\bC^\times \cdot \bar x} \hat \otimes_{\bC[\![\hbar]\!]}
  \hat U((V^{\Gamma_0})^\perp,\Gamma_0)_0[\![c_{i+1},\ldots,c_r]\!], \\
  \hat H(V,\Gamma)_{\bC^\times \cdot \bar x} \cong \hat
  H(V^{\Gamma_0},\{1\})_{\bC^\times \cdot \bar x} \hat \otimes_{\bC[\![\hbar]\!]}
  \Mat_{|\Gamma/\Gamma_0|}(\hat H((V^{\Gamma_0})^\perp,\Gamma_0)_{0})[\![c_{i+1},\ldots,c_r]\!].
\end{gather}
\end{theorem}
Here $\Mat_r(A)$ is the algebra of $r$ by $r$ matrices with
coefficients in $A$ (i.e., $\Mat_r(\bC) \otimes A$). 

The theorem implies that Question \ref{q:q-slice} has a positive
answer for every completed spherical symplectic reflection algebra
$\hat U_{\lambda}(V,\Gamma)$.
% , where $\lambda = (\lambda_1,\ldots,\lambda_r)$,
% $U_{\lambda}(V,\Gamma) :=U(V,\Gamma)/ (c_i-\hbar \lambda_i)$, and
% $\hat U_{\lambda}(V,\Gamma)$ is the $\hbar$-adic completion.
\begin{remark}
  The way that the matrix algebra above actually arises is by the
  centralizer construction: for any algebra $A$ with an action by
  $\Gamma_0 < \Gamma$, one takes $\End_A(\text{Fun}_{\Gamma_0}(\Gamma, A))$,
  where $\text{Fun}_{\Gamma_0}(\Gamma, A)$ is the right $A$-module of
  $\Gamma_0$-equivariant functions from $\Gamma$ to $A$. It is clear
  that $\End_A(\text{Fun}_{\Gamma_0}(\Gamma, A)) \cong
  \Mat_{|\Gamma/\Gamma_0|}(A)$, but the isomorphism depends on a set
  of representatives of the right cosets $\Gamma_0\setminus
  \Gamma$. We apply this to $A = \hat
  H((V^{\Gamma_0})^\perp,\Gamma_0)_0$.
\end{remark}
The proof of the above theorem is by checking that the arguments of
\cite{Los-csra} go through when one completes along the punctured line
$\bC^\times \cdot \bar x$ instead of at $\bar x$. In fact, the above
statement is a corollary of a sheafified statement \cite[Theorem
2.5.3]{Los-csra}, which gives a $\bC^\times$-equivariant isomorphism
of sheaves of $\widehat{\Sym \mathfrak{c}} \rtimes \Gamma$-algebras on
the symplectic leaf containing $\bar x$.  We note that in the case where
$\Gamma < \GL(U)$ for $U \subseteq V$ a Lagrangian subspace (so $V
\cong T^* W$ and hence $\GL(U) < \Sp(V)$, one can instead use the
argument of \cite{BE-pirfrCa}, which is simpler, and again
complete along the punctured line instead of the point. We omit
further details.

\section{The hypertoric case}\label{s:hypertoric}
Let $T^m = (\bC^\times)^m$ act linearly and faithfully on $\bA^n$, and let
$\mathfrak{t}^m := \Lie T^m$.  Associated to this is a Hamiltonian
action of $T^m$ on the cotangent bundle $T^*\bA^n$, with the moment map
$\mu: T^*\bA^n \to (\mathfrak{t}^m)^* \cong \bC^m$.  We consider the
variety $X = \mu^{-1}(0)/\!/T^m$, called the affine hypertoric
variety.  This is also denoted by $T^*\bA^n /\!/\!/\!/ T^m$, the
Hamiltonian reduction. We assume here that $\mu^{-1}(0)/\!/T^m$ admits
a smooth symplectic resolution given by a Geometric Invariant Theory quotient,
$\mu^{-1}(0)/\!/_\chi T^m = \Proj \bigoplus_{m \geq 0}
\bC[\mu^{-1}(0)]^{m \chi}$, for a suitable character $\chi: T^m \to
\bC^\times$; in other words, as we recall in the proof, the action of
$T^m$ is given by a unimodular hyperplane arrangement.

Let us use the $\bC^\times$-action given by dilations in the vector
space $\bC^{2n} = T^*\bA^n$; thus $k=2$ in the previous notation.

In this case, a much stronger statement than Theorem \ref{t:main-sl}
holds: we have in fact an equivariant Darboux-Weinstein decomposition
\emph{Zariski-locally}, as we prove below. 
\begin{theorem} \label{t:hypertoric}
  % Suppose $X$ is a unimodular hypertoric cone.
  Let $X$ be a hypertoric cone as above, $Z$ be a symplectic leaf of
  $X$ other than the vertex, and $z \in Z$ a point. Then there is an
  open $\bC^\times$-stable neighborhood $X^\circ$ of $z$ which splits
  as a product of $\bC^\times$-Poisson varieties,
\[
X^\circ \cong Z^\circ \times S,
\]
with $Z^\circ = X^\circ \cap Z$ and $S$ the hypertoric cone corresponding to
the slice of $Z$.  

Moreover, $Z^\circ$ is $\bC^\times$-Poisson isomorphic to the
complement in $\bA^{\dim Z}$
% $=\Spec[x_1,\ldots,x_{\dim Z/2},y_1,\ldots,y_{\dim Z/2}]$
of $\leq n$ linear hyperplanes, equipped with the standard symplectic
form,
% $\omega = \sum dx_{i} \wedge dy_{i}$,
each of the coordinate functions homogeneous (not necessarily
of degree one), and the
symplectic form having weight two.
% Then for some $s \in S$ there is a decomposition \eqref{e:pr-ass} of
% $\cs$ Poisson schemes.
\end{theorem}
As we will see, the hyperplanes appearing in the second paragraph consist of
some of the coordinate hyperplanes and at most $m/2$ additional linear
hyperplanes.  We note that we use the notation $Z$ for the symplectic
leaf in the theorem (and proof below) so as to be able to use
coordinates $x_i, y_i$ on $T^*\bA^n$ without confusion.

The theorem implies that Question \ref{q:slice} has an affirmative
answer in this case.
\begin{remark}
  For the above theorem, we did not need to assume that $\bC^\times$
  acts freely on $z$, nor even that the stabilizer of $z$ be minimal
  in $Z$.  
  % Indeed, the weights on the coordinate functions on $Z$ which are
  % nonzero at $z$ could have a greatest common divisor
  % which
  % is not one, and then this is the order of the stabilizer of $z$
  % under $\bC^\times$, denoted by $\ell$ in Theorem
  % \ref{t:main-sl}.
  This is a special feature of the hypertoric case (already in the
  finite linear quotient case, it follows from the preceding section
  that there need not be a product decomposition when $\bC^\times$
  acts nonfreely, but rather one only gets a statement as in Theorem
  \ref{t:main-sl}.(iv)). Note that, for generic $z \in Z$, the order
  of the stabilizer in $\bC^\times$ is either $1$ or $2$, just as in the case of
  finite linear quotients (so the $\ell$ appearing in Theorem \ref{t:main-sl}
 will be either $1$ or $2$).
\end{remark}
\begin{remark}
  Recall that there is an important residual Hamiltonian action of
  $\bar T_X := (\bC^\times)^n/T^m$ on $X$, obtained from the standard
  action of $(\bC^\times)^n$ on $\bA^n$ and hence its Hamiltonian
  action on $T^* \bA^n$. We can similarly define residual tori $\bar
  T_Z$ and $\bar T_S$ acting on $Z$ and $S$, respectively.  It follows
  from the proof below that $X^\circ$ and $Z^\circ$ are stable under
  their residual actions.  There is a canonical exact sequence $1 \to
  \bar T_S \to \bar T_X \to \bar T_Z \to 1$ (whose maps descend from the
  inclusion and projection of coordinates on $\bA^n$). Then, the
  isomorphism of the theorem is compatible with the actions of $\bar
  T_S$, and the projection $X^\circ \twoheadrightarrow Z^\circ$
  carries the action of $\bar T_X$ to that of $\bar T_Z$.
\end{remark}
\begin{remark}\label{r:hypertoric-family}
  It follows from the proof below that one similarly has a
  decomposition for the total space $\mathcal{X} := T^*\bA^n /\!/ T^m$ of
  the natural family $\bar \mu: \mathcal{X} \to (\mathfrak{t}^m)^*$ of
  (Poisson) deformations of $X$, with $\bar \mu$ the map which factors
  the moment map $\displaystyle \mu: T^*\bA^n \twoheadrightarrow
  \mathcal{X} \mathop{\to}^{\bar \mu} \mathfrak{t}^m$ (so $X=\bar
  \mu^{-1}(0)$, and $X_t:=\bar \mu^{-1}(t)$ gives the natural
  deformation for $t \in \mathfrak{t}^m$).\footnote{We remark that the
    assumption that $X$ admits a symplectic resolution is equivalent
    to the statement that $X_t$ be smooth for generic $t$, and hence
    smooth affine symplectic.}  Namely, we define $\mathcal{X}^\circ$
  in the same way as we define $X^\circ$ in the proof (except without
  intersecting with $\bar \mu^{-1}(0)$), and we obtain the
  $\bC^\times$-Poisson decomposition
\begin{equation}
\mathcal{X}^\circ \cong Z^\circ \times \mathfrak{t}^* \times \mathcal{S},
\end{equation}
with $\mathfrak{t} = \Lie T$ as defined in the proof below, $Z^\circ$
the same as in theorem, and $\mathcal{S}$ the natural deformation of
the hypertoric variety $S$ appearing there. The Poisson structure on
$\mathfrak{t}^*$ is zero, and its $\bC^\times$ action is the square of
the dilation action on $\bC^{\dim \mathfrak{t}} \cong \mathfrak{t}^*$
(i.e., giving the coordinate functions degree two).  There is also a
natural formula for the map $\bar \mu$ on the RHS compatible with the
isomorphism, and the residual Hamiltonian torus actions are compatible
as in the previous remark.
\end{remark}
\begin{proof}[Proof of Theorem \ref{t:hypertoric}]
%  We (mostly) follow the notation of \cite{PW-ih}. 
  By \cite[3.2,3.3]{BD-gtthm}, the assumption on the resolution of $X$
  above is equivalent to the following unimodularity condition.
  Without loss of generality, assume that $T^m$ acts on standard
  coordinate functions $x_1, \ldots, x_n, y_1, \ldots, y_n$ of $T^*
  \bA^n$ diagonally, i.e., by
\[
(c_1,\ldots,c_m) \cdot x_i = \prod_{j=1}^m c_j^{b_{ij}} x_i, \quad
(c_1, \ldots, c_m) \cdot y_i = \prod_{j=1}^m c_j^{-b_{ij}} y_i.
\]
In other words, the columns of the $n \times m$-matrix $(b_{ji})$ span
the kernel of the matrix dual to the associated central hyperplane
arrangement of $n$ hyperplanes in $(\bC^{n-m})^*$.  The unimodularity
condition is equivalent to the condition 
that all nonzero $m \times m$-minors are $\pm 1$.

% Let $\mu: T^* \bA \to \mathfrak{t}^m = \Lie T^m$ be the associated
% moment map.
In coordinates $\mathfrak{t}^m = \bC^m$ the moment map $\mu$ takes the
form
\[
\mu(x_1,\ldots,y_n)_i = \sum_{j=1}^n b_{ji} x_j y_j.
\]
Let $\pr: \mu^{-1}(0) \to X$ be the projection.  The symplectic leaves
of $X$ are given by certain subgroups $\hat T < T^m$, called
\emph{parabolic} subgroups, which are defined as those such that $\hat
T = \Stab(v)$ for some $v \in \bC^n$. Given such a subgroup, let $F
\subseteq \{1,\ldots,n\}$ be the subset of coordinates on which $\hat
T$ acts \emph{nontrivially}, so that $F^c$, the complement, is the
subset on which $\hat T$ acts trivially. Henceforth, for every $H
\subseteq \{1,\ldots,n\}$, we let $\bA^H$ denote the coordinate subset
corresponding to $H$; then $\bA^{F^c} \subseteq \bA^n$ is the fixed
locus of $\hat T$. Let $T: = T^m / \hat T$, which acts on
$\bA^{F^c}$. One obtains the closure $\bar Z$ of a symplectic leaf $Z$
by the Hamiltonian reduction of $T^*\bA^{F^c}$ by $T$, i.e., $\bar Z =
T^*\bA^{F^c}/\!/\!/\!/ \hat T = (\mu^{-1}(0) \cap T^*\bA^{F^c}) /\!/
T$.  This identifies with $\pr(T^*\bA^{F^c} \cap \mu^{-1}(0))
\subseteq X$. As before, we will also let $\mathfrak{t} = \Lie(T)$ and
$\hat {\mathfrak{t}} = \Lie(\hat T)$.

It is a general fact that the closures of the symplectic leaves are of
this form. There are only finitely many, as there are only finitely
many parabolic subgroups of $T^m$. This therefore determines the
symplectic leaves themselves: the leaf $Z$ is the complement in $\bar
Z$ of all proper subsets which are the closures of symplectic leaves.
However, multiple $\hat T$ can produce the same symplectic leaf. We
will therefore assume that $\hat T$ is maximal for its leaf.  Then,
the leaf $Z$ is the image under
$\pr$ of the union of the closed free $T$-orbits.  (Moreover, under this
assumption, one can easily check that this produces a bijection
between maximal $\hat T$ and symplectic leaves; see \cite[\S 2]{PW-ih},
where this is explained in terms of $F \subseteq \{1,\ldots,n\}$, and
the $F$ that occur in this way are the \emph{coloop-free flats}, which
are natural combinatorially defined subsets of $\{1,\ldots,n\}$. We
will not use these facts here.)

Let $Z$ be a symplectic leaf, $\hat T < T^m$ be a (in fact, the unique)
maximal corresponding parabolic subgroup, and $F\subseteq
\{1,\ldots,n\}$ be the subset as above of coordinates on which $\hat
T$ acts nontrivially.  The closure $\bar Z$ is then given by $\bar Z =
(\mu^{-1}(0) \cap T^*\bA^{F^c}) /\!/ T$.  Let $z \in Z$ and let $\tilde
z \in T^*\bA^{F^c}$ be a preimage of $z$ in $\mu^{-1}(0) \cap
T^*\bA^{F^c}$. 

Let $G \subseteq F^c$ be a subset with $|G| = \dim T$ such that $T$
acts with finite kernel on $\bA^G$, such that for all $i \in G$,
either $\tilde z(x_i) \neq 0$ or $\tilde z(y_i) \neq 0$.  Such a
subset must exist because $T \cdot \tilde z$ is a free orbit.
Moreover, by our unimodularity hypothesis, the kernel of $T$ on
$\bA^G$, \`a priori finite, must actually be trivial; thus, $T$ acts
generically faithfully on $\bA^G$. Without loss of generality (up to
swapping some of the $x_i$ with $y_i$), we can actually assume that
$\tilde z(x_i) \neq 0$ for all $i \in G$.  

We then define $(T^* \bA^n)^\circ$ as the complement of coordinate
hyperplanes on which $\tilde z$ is nonzero.  We similarly define $(T^*
\bA^H)^\circ := T^* \bA^H \cap (T^* \bA^n)^\circ$ and $(\bA^H)^{\circ}
:= \bA^H \cap (T^* \bA^n)^\circ$ for all $H \subseteq
\{1,\ldots,n\}$. In particular, $(\bA^G)^\circ$ is the locus where all
coordinates are nonzero, and $T$ acts freely and transitively on it.
It follows that the $T$-orbits in $(T^* \bA^{F^c})^\circ$, or
equivalently the $T^m$-orbits, are all closed and free, and that
$Z^\circ = \pr((T^* \bA^{F^c})^\circ \cap \mu^{-1}(0))$.  (We remark
that, if one desired, one could alternatively have defined
$(T^*\bA^{F^c})^\circ$ to be the maximal subset consisting of closed
$T$-orbits such that the $x_i$ coordinate is nonzero for all $i \in
G$; this is a larger open subset and depends only on $G$ and $Z$ and
not on $z$, but has the disadvantage of not being affine. We could
similarly consider any (affine) open $T$-stable subset thereof. We
would then set $T^*\bA^n = (T^*\bA^{F^c})^\circ \times T^* \bA^F$.)

Next, let us restrict the moment map to $T^*(\bA^{F^c})^{\circ}$.
This lands in $\mathfrak{t}^* \subseteq (\mathfrak{t}^m)^*$. By the
above, for every $z \in \mathfrak{t}^*$, the equations cutting out
$\mu^{-1}(z)$ in $(\bA^{F^c})^\circ$ uniquely solve for $x_i y_i, i \in G$
in terms of $x_j y_j, j \in F^c \setminus G$.  In other words, since $x_i$ are invertible on $(\bA^{F^c})^\circ$ for $i \in G$, these equations uniquely
solve for $y_i, i \in G$ in terms of the other coordinates. Hence,
% defining $(T^*\bA^H)^\circ := (T^* \bA^n)^{\circ} \cap T^*\bA^H$ and
% similarly $(\bA^H)^\circ := \bA^H \cap (T^* \bA^n)^\circ$,
we have an isomorphism 
\[
\pi^{F^c \setminus G} \times \mu \times \pi': T^*(\bA^{F^c})^{\circ}
\mathop{\to}^\sim (T^*\bA^{F^c \setminus G})^\circ
\times \mathfrak{t}^* \times (\bA^G)^\circ,
\]
with $\pi^{F^c \setminus G}$ the projection to $T^*\bA^{F^c \setminus
  G}$ and $\pi'$ the projection to $\bA^G \subseteq T^* \bA^G$ (just
the $x$ coordinates). Since $(\bA^G)^\circ \cong T$ consists of a
single free $T$-orbit and $(T^*\bA^{F^c})^\circ$ is affine, the above
produces an isomorphism $(T^* \bA^{F^c \setminus G})^\circ \cong
Z^\circ$.

The final assertion then follows immediately. The weights of the
coordinate functions $x_i, y_i, i \in F^c \setminus G$ on
$T^*\bA^{F^c\setminus G} \cong \bA^{\dim Z}$ are then explicitly given
as follows: for each $i \in F^c \setminus G$, let $r_{ij} \in \bZ$ be
the unique integers such that $x_i \prod_{j \in G} x_j^{r_{ij}}$ is
$T$-invariant. Then, $|x_i| := |x_i \prod_{j \in G} x_j^{r_{ij}}| = 1
+ \sum_{j \in G} r_{ij}$, and similarly we compute the degrees of the
$y_i$.

Next, write $T^m = T' \times \hat T$ for $T' \subseteq T^m$ some
connected subtorus such that the composition $T' \to T^m \to T$ is an
isomorphism.  Let $\mu_{T'}$ and $\mu_{\hat T}$ be the restricted
moment maps, valued in $(\mathfrak{t}')^*$ and $\hat
{\mathfrak{t}}^*$, respectively.  Then $\mu^{-1}(0) =
\mu_{\mathfrak{t}'}^{-1}(0) \cap \mu_{\hat {\mathfrak{t}}}^{-1}(0)$.
Let $\pi^F: T^*\bA^n \to T^*\bA^F$ be the projection.  Then the
projection to $T^*\bA^{F^c}$ produces an isomorphism 
\[
(\pi^F)^{-1}(z)
\cap \mu_{\mathfrak{t}'}^{-1}(0) \cong
\mu^{-1}(-\mu_{\mathfrak{t}'}(z)) \cap T^* \bA^{F^c},
\] where we use
in the first term on the RHS the isomorphism $\mathfrak{t}' \cong
\mathfrak{t}$.  Thus, restricted to $(T^*\bA^n)^{\circ} \cap
\mu_{\mathfrak{t}'}^{-1}(0)$, $\pi^F$ is a fibration with fibers
isomorphic to $T^*\bA^{F^c \setminus G}  \times
(\bA^G)^{\circ}$.  We conclude that $\pi^F \times \pi^{F^c \setminus
  G} \times \pi'$ induces an isomorphism
\[
\mu^{-1}(0) \cap (T^* \bA^n)^{\circ} \cong (\pi^F \times \pi^{F^c
  \setminus G} \times \pi')(\mu_{\hat {\mathfrak{t}}}^{-1}(0)) \cap
(T^* \bA^F \times (T^* \bA^{F^c\setminus G})^{\circ} \times (\bA^G)^{\circ}).
\]
Now, $\mu_{\hat {\mathfrak{t}}}$ is trivial in the $T^*\bA^{F^c}$ direction. That is,
 for $\mu^F_{\hat {\mathfrak{t}}}: T^*\bA^F \to \hat {\mathfrak{t}}$ the restricted moment map, we have
 $\mu_{\hat {\mathfrak{t}}} = \mu^F_{\hat {\mathfrak{t}}} \circ \pi^F$. We conclude that
\[
(T^*\bA^n)^\circ \cap \mu^{-1}(0) \cong ((\mu^F_{\hat {\mathfrak{t}}})^{-1}(0) \cap T^* \bA^F) \times (T^*\bA^{F^c \setminus G})^\circ \times (\bA^G)^{\circ}.
\]

Moreover, given functions on any two of the above three factors, if we
write them in terms of standard coordinate functions, the resulting
functions on $T^* \bA$ Poisson-commute.  By the unimodularity
condition, the action of $T$ on $(\bA^G)^\circ$ is free and
transitive. Therefore, we can replace each coordinate function $x_i$ or $y_i$
on $T^*\bA^F$ with its product by a monomial $x_i' := x_i \prod_{j \in
  G}x_j^{r_j}$ (or similarly $y_i' := y_i \prod_{j \in G}x_j^{s_j}$)
so that $x_i', y_i'$ are invariant under $T'$, where here the
$x_j$ are the coordinate functions in the $\bA^n$ direction (hence
when $j \in G$, they are coordinates on $\bA^G$).  The new $x_i',
y_i'$ still commute with all functions on $T^* \bA^{F^c \setminus G}
\times \bA^G$ when written in terms of standard coordinate functions.
Therefore, restricting to $\mu^{-1}(0)$ and quotienting by $T^m$
yields the $\bC^\times$-equivariant Poisson isomorphism
\[
(T^* \bA)^\circ /\!/\!/\!/ T^m \cong Z^\circ \times S. \qedhere
\]
% and by construction (each projection map was an inclusion of Poisson
% subalgebras), this map is Poisson.
\end{proof}

\subsection{Quantization}
Since we have a Zariski-local decomposition, we can immediately
conclude a quantization of $X^\circ$.  Let $A_Z$ and $A_{X}$ be the
standard deformation quantizations of the hypertoric cones $Z$ and
$S$.  Namely, these are given by quantum Hamiltonian reduction $A_Z =
\caD_{\hbar}(A^{F^c})/\!/\!/\!/ T :=
(\caD_{\hbar}(\bA^{F^c})/\mu_{q,F^c}(\mathfrak{t})
\caD(\bA^{F^c}))^T$, with $\mu_{q,F^c}: \mathfrak{t} \to
\caD_\hbar(\bA^{F^c})$ the quantum comoment map with
$\mu_{q,F^c}(\xi)$ the vector field given by the action of $\xi$,
viewed as a differential operator. Similarly define $A_{S}$ and $A_X$.
Since $X^\circ$ is the complement in $X$ of the images of some linear
hyperplanes, we can define the corresponding localization
$A_{X^\circ}$ by inverting the corresponding linear functions and
taking the $\hbar$-adic completion.  Note that $Z^\circ = Z \cap
X^\circ$; thus $A_{Z^\circ}$ is obtained from $A_Z$ by inverting the
same coordinate functions and completing.  We equip these algebras
with the Euler vector field which assigns the coordinate functions the
same degrees as in the graded algebras $\cO(S), \cO(Z)$, and $\cO(X)$.
Then $A_{Z^\circ} \hat \otimes_{\bC[\![\hbar]\!]} A_{S}$ is a
deformation quantization of $X^\circ$, which is
$\bC^\times$-compatible by Theorem \ref{t:hypertoric}.  (Note that
$\cO(Z)$ and $\cO(S)$ are not, in general, nonnegatively graded;
neither is, obviously, $A_{X^\circ}$ nor $A_{Z^\circ}$.  Thus, the
quantizations are not completed Rees algebras of nonnegatively
filtered algebras.)
\begin{theorem}\label{t:q-hypertoric}
We have an isomorphism of $\bC^\times$-compatible deformation
quantizations,
\begin{equation}
  A_{X^\circ}  \cong
  A_{Z^\circ} \hat  \otimes_{\bC[\![\hbar]\!]} A_{S}.
\end{equation}
\end{theorem}
The fact that both sides of the isomorphisms above are quantizations
of $X^\circ$ is immediate; to prove the two are isomorphic follows
from a straightforward quantum analogue of the proof of Theorem
\ref{t:hypertoric}.  We note that the theorem yields a positive answer
to Question \ref{q:q-slice} in the case of the quantization
$A_{X^\circ}$.
\begin{remark}
  In the previous section, we considered the universal deformation of
  algebras of invariant differential operators, given by spherical
  symplectic reflection algebras.  One can similarly form the
  universal family of (graded) deformation quantizations of the
  quantized hypertoric varieties $A_X$ above, by performing quantum
  Hamiltonian reduction not merely at the $0$ character but at
  arbitrary characters of $\mathfrak{t}^m$ (these were considered, for
  example, in \cite{BLPW-HCO}).  Then, the above theorem goes through
  for these as well: given a character $\zeta \in (\mathfrak{t}^m)^*$,
  one considers the restriction $\zeta|_{\hat{\mathfrak{t}}}$ as in
  the proof of Theorem \ref{t:hypertoric}, and obtains the
  $\bC^\times$-equivariant (continuous $\bC[\![\hbar]\!]$-algebra)
  isomorphism:
\begin{equation}
  A^{\hbar \zeta}_{X^\circ} \cong
  A_{Z^\circ} \hat  \otimes_{\bC[\![\hbar]\!]} A_{S}^{\hbar \zeta|_{\hat{\mathfrak t}}},
\end{equation}
where $A_{X}^{\hbar \zeta} := (\caD_\hbar(\bA^{n})/(\mu_{q}(t)-\hbar
\zeta(t))_{t \in \mathfrak{t}^m} \caD_\hbar(\bA^{n}))^{T^m}$,
$A_{X^\circ}^{\hbar \zeta}$ is its completed localization, and we
similarly define $A_{S}^{\zeta|_{\hat{\mathfrak t}}}$. Note that we
need not deform $A_{Z^\circ}$; this makes sense since $Z^\circ$ is
just the complement of some hyperplanes in a symplectic affine space.
We conclude that Question \ref{q:q-slice} has a positive answer in the
case of these quantizations.  (One can also put the deformations
together into a family and produce an isomorphism resembling that of
Remark \ref{r:hypertoric-family} and analogous to Theorem
\ref{t:sra-gc}).
\end{remark}
\begin{remark}
  It would also be interesting to prove a sheafified version of the
  decomposition, parallel to the decomposition of \cite[Theorem
  2.5.3]{Los-csra}, which produces a global statement for sheaves on
  $Z$ of completed quantizations, and to study how the decomposition
  varies as $z$ (or the line $\bC^\times \cdot z$) varies.
  % in this context. For the sheafified statement, we expect that, as
  % in \cite{Los-csra}, one needs to twist the sheaves that appear.
\end{remark}

\bibliographystyle{amsalpha}

\def\cprime{$'$} \def\cprime{$'$} \def\cprime{$'$} \def\cprime{$'$}
  \def\cprime{$'$}
\providecommand{\bysame}{\leavevmode\hbox to3em{\hrulefill}\thinspace}
\providecommand{\MR}{\relax\ifhmode\unskip\space\fi MR }
\providecommand{\MRhref}[2]{%
  \href{http://www.ams.org/mathscinet-getitem?mr=#1}{#2}
}
\providecommand{\href}[2]{#2}

%\bibliography{../bibtex/master}
\end{document}